\documentclass[A4paper, 12pt]{article} 
\usepackage[T1]{fontenc}
\usepackage[utf8]{inputenc} 
\usepackage[all,cmtip]{xy}
\usepackage[new]{old-arrows}
\usepackage{extarrows}

\usepackage[margin=1in, top=1.2in, bottom=1.4in]{geometry}

\usepackage{graphicx} 

\usepackage{booktabs} 
\usepackage{array} 
\usepackage{paralist} 
\usepackage{verbatim} 
\usepackage{subfig} 
\usepackage{mathtools}
\usepackage{amssymb}
\usepackage{changepage}
\usepackage{amsmath}
\usepackage{amsthm}
\usepackage{textcomp}
\usepackage{mathdots}
\usepackage{mathrsfs}
\usepackage{enumitem}
\usepackage[colorlinks, allcolors=blue]{hyperref}
\usepackage{appendix}
\usepackage{stmaryrd}
\usepackage{hyperref}
\usepackage[bottom]{footmisc}
\hypersetup{colorlinks,linkcolor={blue},citecolor={blue},urlcolor={black}} 

\usepackage{subfiles}

\usepackage[tightpage,psfixbb]{preview}
\usepackage{titlesec}

\titleformat{\section}
  {\Large} 
  {\thesection} 
  {0.5em} 
  {} 

\titleformat{\subsection}[runin] 
  {\normalfont\bfseries} 
  {\thesubsection} 
  {0.5em} 
  {} 

\titlespacing{\subsection} 
  {0pt} 
  {10pt} 
  {0.5em} 

\titleformat{\subsubsection}[runin]
  {\normalfont\itshape} 
  {\thesubsubsection} 
  {0.5em} 
  {}

\titlespacing{\subsubsection}
  {0pt} 
  {10pt} 
  {0.5em} 

\usepackage[nottoc,notlof,notlot]{tocbibind} 
\usepackage[titles,subfigure]{tocloft} 
\theoremstyle{sltheoremstyle}
\newtheorem{theorem}{Theorem}[section]
\newtheorem{lemma}[theorem]{Lemma}

\newtheorem{corollary}[theorem]{Corollary}

\newtheorem{proposition}[theorem]{Proposition}
\newtheorem{conjecture}[theorem]{Conjecture}

\newtheorem*{corollary-non}{Corollary}
\newtheorem*{lemma-non}{Lemma}
\newtheorem*{theorem-non}{Theorem}
\newtheorem*{proposition-non}{Proposition}
\newtheorem*{condition-non}{Condition}
\newtheorem*{conditions-non}{Conditions}

\theoremstyle{definition}

\newtheorem{remark}[theorem]{Remark}

\newtheorem{definition}[theorem]{Definition}

\newtheorem{example}[theorem]{Example}

\newtheorem{examples}[theorem]{Examples}

\usepackage[x11names]{xcolor}
\usepackage{tabularx}
\usepackage{geometry}
\usepackage{setspace}
\usepackage{lipsum}
\usepackage{authblk}

\usepackage[backend=bibtex,style=alphabetic]{biblatex}

\usepackage{tikz} 
\usepackage{tikz-cd}

\newcommand{\PSL}{\textnormal{PSL}}

\newcommand{\set}[1]{\left\{ #1 \right\}}
\newcommand{\va}[1]{\left| #1 \right|}

\newcommand{\Isom}{\textnormal{Isom}}

\newcommand{\etale}{\textnormal{\'et}}

\newcommand{\Id}{\textnormal{Id}}

\newcommand{\Homeo}{\textnormal{Homeo}}

\newcommand{\Stab}{\textnormal{Stab}}
\newcommand{\sm}{\setminus}

\newcommand{\CC}{\mathbb{C}}

\newcommand{\OO}{\mathcal{O}}

\newcommand{\PP}{\mathbb{P}}

\newcommand{\PGL}{\textnormal{PGL}}

\newcommand{\et}{\textnormal{\'et}}

\newcommand{\SL}{\textnormal{SL}}

\newcommand{\RR}{\mathbb{R}}

\newcommand{\ZZ}{\mathbb{Z}}

\newcommand{\Ker}{\textnormal{Ker}}

\newcommand{\Gal}{\textnormal{Gal}}

\newcommand{\Aut}{\textnormal{Aut}}

\newcommand{\Spec}{\textnormal{Spec}}

\newcommand{\ca}[1]{{\mathcal{#1}}}
\newcommand{\bb}[1]{{\mathbb{#1}}}
\newcommand{\msf}[1]{{\mathsf{#1}}}
\newcommand{\mr}[1]{{\mathscr{#1}}}

\let\rm\relax 
\newcommand{\rm}[1]{{\mathrm{#1}}}

\DeclareSymbolFont{bbm}{U}{bbm}{m}{n}
\DeclareSymbolFontAlphabet{\mathbbm}{bbm}

\newcommand{\vXR}{\va{\ca X(\RR)}}

\begin{document}

\title{\Large{\textbf{On the topology of real algebraic stacks}}}

\author{Emiliano Ambrosi and Olivier de Gaay Fortman \\ \today}

\date{}

\onehalfspacing

\maketitle 
\abstract{\noindent
\emph{Motivated by questions arising in the theory of moduli spaces in real algebraic geometry, we develop a range of methods to study the topology of the real locus of a Deligne--Mumford stack over the real numbers. As an application, we verify in several cases the Smith--Thom type inequality for stacks that we conjectured in an earlier work. This requires combining techniques from group theory, algebraic geometry, and topology.
}

\section{Introduction}
This work is the second in a two-part series dedicated to the study of the topology of real Deligne–Mumford stacks. In the first paper \cite{AdGF-grup}, we established several foundational topological properties of such stacks and formulated a conjecture on their cohomology, extending the classical Smith–Thom inequality to the stack-theoretic setting. In the present paper, we develop techniques for computing the topology of various classes of real Deligne–Mumford stacks -- such as finite quotient stacks and gerbes over real varieties -- and apply these methods to verify the conjecture in a range of cases.
\subsection{Smith--Thom inequality for real algebraic varieties.} We begin by recalling the conjecture and our motivation behind it. Let $X$ be a real algebraic variety, by which we mean a reduced and separated scheme of finite type over $\RR$. 
One of the foundational results in real algebraic geometry (see \cite{floydperiodic, borel-seminar, thom-homologie, itenberg-enriques, mangolte} for various proofs) is the Smith--Thom inequality
\begin{align} \label{align:inequality-ST}
h^*(X(\RR))= 
\dim \rm{H}^\ast(X(\RR),\ZZ/2) \leq 
\dim \rm{H}^\ast(X(\CC), \ZZ/2)= h^*(X(\CC)).
\end{align}
It allows one to bound the cohomology of $X(\RR)$ in terms of the cohomology of $X(\CC)$, usually much easier to compute. Here, and in the sequel, $h^\ast(Y)$ denotes the dimension of the cohomology ring $\rm H^\ast(Y,\ZZ/2) = \oplus_{i \geq 0} \rm H^i(Y,\ZZ/2)$ of a topological space $Y$. 

\subsection{Smith--Thom inequality for real algebraic stacks.}
In recent years, there has been increasing interest in moduli problems over $\RR$ (see e.g.\ 
\cite{gross-harris, seppalasilhol, cohomology-M0n}), particularly in determining whether \eqref{align:inequality-ST} attains equality for the associated coarse moduli space. Notable cases include moduli spaces of stable vector bundles on a curve \cite{brugalleschaffhauser-2022}, Hilbert schemes of points \cite{fu2023maximalrealvarietiesmoduli, kharlamov2024unexpectedlossmaximalitycase, kharlamov2025smiththomdeficiencyhilbertsquares}, and symmetric powers of varieties \cite{biswasmello-2017, franz-2018}. Note, however, that such a study says something about the \emph{real moduli space} associated to the moduli problem only if this real moduli spaces arises as the real locus of the coarse moduli space, a phenomenon which in fact seems rare. For instance, if $\msf A_1$ is the coarse moduli space of elliptic curves, then $\msf A_1(\RR) = \RR$ parametrizes complex elliptic curves that admit a real structure up to complex isomorphism, whereas the real moduli space of real elliptic curves has two connected components (there are exactly two topological types of real elliptic curves).

To bypass this limitation, and start a 
systematic approach to study the topology of real moduli spaces, one is led to consider real algebraic stacks. 
If $\ca X$ is such a stack, then $ \mathcal X(\mathbb R) $ is a category rather than a set. To obtain a topological space in a way that generalizes the euclidean topology on $X(\mathbb R)$ when $X$ is a real variety, one considers the set $\va{ \mathcal X(\mathbb R)}$ of isomorphism classes of the category $\mathcal  X(\mathbb R)$, and defines a natural topology on $\vert \mathcal X(\mathbb R)\vert $
as in \cite{degaay-realmoduli}. A similar procedure defines a topology on the set $\va{\ca X(\CC)}$ of isomorphism classes of $\ca X(\CC)$ (if $\ca X$ is separated Deligne--Mumford, the latter coincides with the topology on $\va{\ca X(\CC)} \simeq M(\CC)$ induced by the coarse moduli space $\ca X \to M$).

The advantage of this perspective is that when the algebraic stack 
$\ca X$ represents a moduli problem—parametrizing equivalence classes of certain algebraic objects (such as genus 
$g$ curves, or sheaves on a fixed variety)—the set 
$\va{\ca X(\RR)}$ corresponds to the real isomorphism classes of the real objects. For instance, 
$\va{\ca M_g(\RR)}$  represents the space of isomorphism classes of real algebraic curves of genus $g$. 

In order to understand the topological properties of $\va{\ca X(\RR)}$, a natural first step is to understand whether the foundational inequality (\ref{align:inequality-ST}) generalizes to this setting. As already mentioned in \cite{AdGF-grup}, this is not the case, as e.g.\ the moduli space $\mathcal A_1$ of elliptic curves example shows: we have $h^*(\vert \mathcal A_1(\mathbb C)\vert)=1$ while $h^*(\vert \mathcal A_1(\mathbb R)\vert)=2$.

The main challenge in extending the Smith–Thom inequality \eqref{align:inequality-ST} to algebraic stacks is that, although \(\lvert \mathcal{X}(\mathbb{C}) \rvert\) is equipped with an involution  
$
\sigma \colon \lvert \mathcal{X}(\mathbb{C}) \rvert \to \lvert \mathcal{X}(\mathbb{C}) \rvert
$  
which generalizes complex conjugation on the complex locus of a real variety, the natural map  
\begin{align}\label{align:real-complex-map}
\lvert \mathcal{X}(\mathbb{R}) \rvert \to \lvert \mathcal{X}(\mathbb{C}) \rvert^\sigma
\end{align}
is, even in simple cases, neither injective (Example \ref{ex:A1modZ2}) nor surjective (Example \ref{ex:A1modZ2xZ2}). 

The failure of surjectivity of \eqref{align:real-complex-map} is due to the existence of isomorphism classes of objects $x \in \ca X(\CC)$ which are isomorphic to their complex conjugate, but not defined over $\RR$.  The failure of injectivity is measured by the following observation: for $x \in \ca X(\RR)$, the fiber of \eqref{align:real-complex-map} above the image of $x$ in $ \lvert \mathcal{X}(\mathbb{C}) \rvert^{\sigma}$ is in canonical bijection with the first non-abelian Galois cohomology group $\rm  H^1(G, \Aut(x_\CC))$, where
     $$G \coloneqq \Gal(\CC/\RR) \simeq \ZZ/2.$$
This reveals that the complex locus \(\lvert \mathcal{X}(\mathbb{C})\vert\) of a real stack $\ca X$ is in a sense too small to fully encode all 
information on the topology of
\(\lvert \mathcal{X}(\mathbb{R}) \rvert\), as it
does not capture 
the automorphisms of objects in \(\mathcal{X}(\mathbb{C})\). 
In \cite{AdGF-grup} we took these stabilizer groups into account by considering the inertia stack 
$\mathcal{I}_{\mathcal{X}}$, whose complex locus $\ca I_{\ca X}(\CC)$ consists of pairs \((x, \phi)\) with \(x \in \mathcal{X}(\mathbb{C})\) and \(\phi\) an automorphism of \(x\). 
We conjectured 
the following generalization of the Smith–Thom inequality \eqref{align:inequality-ST} to real Deligne–Mumford stacks.
\begin{conjecture}[cf.\ \cite{AdGF-grup}, Conjecture 1.7] \label{conj:ST}
Let $\ca X$ be a separated Deligne--Mumford stack of finite type over $\RR$, with inertia $\ca I_{\ca X} \to \ca X$. 
Then the following inequality holds:
\begin{align}\label{align:inequality-ST-stacks}
\dim \rm{H}^\ast(\va{\ca X(\RR)},\ZZ/2) \leq 
\dim \rm{H}^\ast( \vert \ca I_{\ca X}(\CC)\vert, \ZZ/2). 
\end{align}
\end{conjecture}
Recall that if $\mathcal X$ is a scheme, then $\ca I_{\mathcal X}\rightarrow \mathcal X$ is an isomorphism. Thus, in that case,  (\ref{align:inequality-ST-stacks}) reduces to the classical Smith--Thom inequality (\ref{align:inequality-ST}). 

We warn the reader that, in general, there is no natural embedding of $\va{\ca X(\RR)}$ into $\vert \ca I_{\ca X}(\CC)\vert$. 
For example, take an elliptic curve $E$ over $\RR$ such that $h^\ast(E(\RR)) = 4$, and consider the stacky quotient $\mathcal X \coloneqq [E/ \langle -1 \rangle ]$, where $-1 \colon E\rightarrow E$ is the multiplication by $-1$. One can show (see Example  \ref{ex:exampleabelianvarieties}) that $\vert \mathcal X(\mathbb R)\vert\simeq [0, 1]\coprod [0, 1]\coprod [0, 1]\coprod [0, 1]$, and that $\vert \ca I_{\ca X}(\CC)\vert \simeq \mathbb P^1(\mathbb C)\coprod \left(\coprod_{x\in E(\mathbb C)[2]}\{x\})) \right)$. The situation is depicted in Figure \ref{fig: E:pm1} below. Although there is no natural embedding $E(\RR) \coprod E(\RR) \to \va{\ca I_{[E/\langle -1 \rangle ]}(\CC)}$, inequality (\ref{align:inequality-ST-stacks}) holds in this case: we have $h^\ast(\va{\ca X(\RR)} = 4$ and $h^\ast(\vert \ca I_{\ca X}(\CC)\vert) = 2 + 4 = 6$. 
\begin{center}
\begin{figure}[h!]
\begin{center}
\begin{picture}(0,130)
\put(-115,6){$\vert [E/\langle - 1 \rangle](\mathbb R)\vert$}
\put(90,6){$\vert \mathcal I_{[E/\langle -1 \rangle]}(\mathbb C)\vert$}
\put(-170,15){\includegraphics[width=0.8\textwidth]{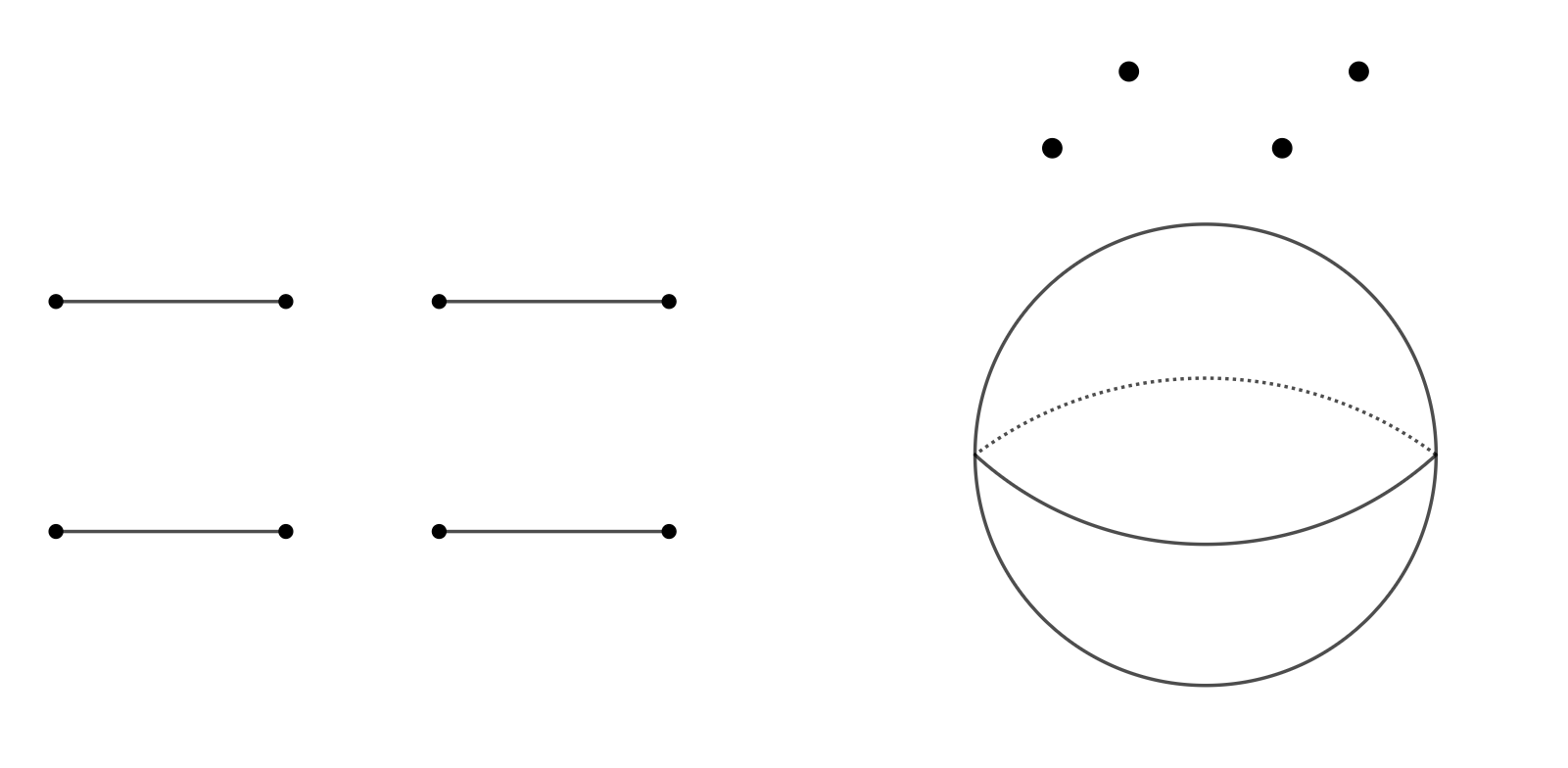}}
\end{picture}
\caption{The real locus and inertia of the stack $[E/\langle -1 \rangle]$.}
\label{fig: E:pm1}
\end{center}
\end{figure}
\end{center}
\subsection{General topological properties of real DM stacks and applications to the Smith--Thom inequality.}
A distinctive feature of the classical Smith–Thom inequality (\ref{align:inequality-ST}) is its inherently global nature: since varieties are locally contractible, the inequality holds trivially at the local level. In contrast, the inequality proposed in Conjecture \ref{conj:ST} does not seem locally trivial. Indeed, any separated Deligne--Mumford stack over $\RR$ is \'etale locally isomorphic to a quotient stack $[X/\Gamma]$, where $X$ is a real algebraic variety and $\Gamma$ is a finite group scheme $\Gamma$ over $\RR$ acting on $X$ over $\RR$, see 
Lemma \ref{lemma:local-structure}. 
Even for $[X/\Gamma]$, Conjecture \ref{conj:ST} does not appear straightforward: the topology of a real quotient stack can be rather complicated as the following theorem shows. 

Let $\Gamma$ be a finite group scheme over $\RR$, so that $G$ acts on $\Gamma(\mathbb C)$ via an involution $\sigma \colon \Gamma(\CC) \to \Gamma(\CC)$, and let $X$ be a quasi-projective scheme over $\RR$ on which $\Gamma$ acts. Choose a complete set of representatives $H \subset \Gamma(\CC)$ for the non-abelian Galois cohomology group
$\rm H^1(G,\Gamma(\mathbb C))=\{\gamma\in \Gamma(\CC) \text{ such that $\gamma\sigma(\gamma)=e$}\}/_\sim$
where $\sim$ is the equivalence relation $\gamma \sim \beta \gamma \sigma(\beta)^{-1}$ for $\beta \in \Gamma(\CC)$. Via a twisting procedure, one can associate to each $\gamma\in H$ a finite group scheme $\Gamma_{\gamma}$ and a quasi-projective scheme $X_{\gamma}$ over $\RR$, as well as a natural action of 
$\Gamma_{\gamma}$ on $X_{\gamma}$ over $\RR$ (see Section \ref{sec : twistingdefinition} for details). 

\begin{theorem}\label{thm:introlemme-prefere}
Consider the above notation. In particular, $X$ is a quasi-projective scheme over $\RR$ endowed with the action of a finite group $\Gamma$, and we let $[X/\Gamma]$ denote the associated quotient stack. Then there is a canonical homeomorphism
    \begin{align*}
    \va{[X/\Gamma](\RR)} \xlongrightarrow{\sim} \coprod_{\gamma \in H} X_\gamma(\RR)/\Gamma_\gamma(\RR).
    \end{align*}
\end{theorem}
Theorem \ref{thm:introlemme-prefere} has a number of positive implications when it comes to the Smith--Thom inequality for stacks. To start with, it has the following consequence. 

 \begin{corollary}
 Let $X$ be a real variety and consider the stack $\ca X = [(X \times_\RR X) / (\ZZ/2)]$, where $\ZZ/2$ acts on $X\times_\RR X$ by permuting the factors. Then Conjecture \ref{conj:ST} holds for $\ca X$. 
 \end{corollary}

 In Section \ref{sec:smiththom-application-quotient}, we use Theorem \ref{thm:introlemme-prefere} to prove Conjecture \ref{conj:ST} in a several other examples. 
Moreover, by combining Theorem \ref{thm:introlemme-prefere} with a lemma that computes the étal-local structure of real Deligne--Mumford stacks 
(see Lemma \ref{lemma:local-structure}), we 
prove the following general result on the topology of real algebraic stacks. 

\begin{theorem} \label{thm:coarse-map-closed}
        Let $\ca X$ be a separated Deligne--Mumford stack of finite type over $\RR$, with coarse moduli space $\ca X \to M$. Then the  morphism on real loci $\va{\ca X(\RR)} \to M(\RR)$ is closed. 
\end{theorem}
In turn, we use Theorem \ref{thm:coarse-map-closed} to prove 
the Smith--Thom inequality \eqref{align:inequality-ST-stacks} for proper real Deligne--Mumford stacks whose non-trivial stabilizer locus is discrete:

\begin{theorem}\label{thm : discrete}
Let $\mathcal X$ be a proper Deligne--Mumford stack of finite type over $\RR$, whose locus of points with non-trivial stabilizer is discrete. 
Then Conjecture \ref{conj:ST} holds for $\ca X$. 
\end{theorem}

The above theorem applies e.g.\ to Kummer-style quotients $[A/ (\mathbb Z/2)]$, where $A$ is a real abelian variety and $\mathbb Z/2$ acts on $A$ via the inversion $-1 \colon A \to A$, see Example \ref{ex:exampleabelianvarieties}.


\subsection{Smith--Thom for stacky curves.}\label{subsec:intro:quotient}

Another corollary of Theorem \ref{thm : discrete} is that it implies Conjecture \ref{conj:ST} for stacky quotients $\ca X = [X/\Gamma]$ of a smooth proper real algebraic curve $X$ by a finite $\RR$-group scheme $\Gamma$ 
acting faithfully on $X$. 
We enhance this result in 
Section \ref{sec:STforcurves}, where we prove Conjecture \ref{conj:ST} for stacky quotients of real curves by finite group schemes that are abelian or act faithfully:
 
\begin{theorem} \label{thm:introST-curves}
Let $X$ be a smooth real curve, and let $\Gamma$ be a finite $\RR$-group scheme, equipped with an action on $X$ over $\RR$.  
Assume that $\Gamma$ is abelian, or that the action is faithful. 
Then Conjecture \ref{conj:ST} holds for the quotient stack $\ca X = [X/\Gamma]$. 
\end{theorem}
Here, a \emph{real curve} is a one-dimensional variety over $\RR$, not necessarily proper (cf.~Section \ref{sec:notation}).
The proof of Theorem \ref{thm:introST-curves} is somewhat indirect, as we do not directly compare the topological spaces \( \lvert [X/\Gamma](\mathbb{R}) \rvert \) and \( \lvert \mathcal{I}_{[X/\Gamma]}(\mathbb{C}) \rvert \). Instead, we compute \( h^*(\lvert [X/\Gamma](\mathbb{R}) \rvert) \) and \( h^*(\lvert \mathcal{I}_{[X/\Gamma]}(\mathbb{C}) \rvert) \) separately by combining local and global methods; one of the main ingredients in the argument is the study of the local structure of the natural morphism $\vert [X/\Gamma](\mathbb R)\vert \rightarrow (X/\Gamma)(\mathbb R)$, which we use to analyze the global geometry of $\vert [X/\Gamma](\mathbb R)\vert$. 

\subsection{Topology of split gerbes over a real variety.} 
Finally, we go beyond the case of quotients of a real variety by a finite group scheme by considering split gerbes over a real variety. These are stacks of the form $\ca X = [U/H]$ where $U$ is a real variety and $H \to U$ a finite \'etale group scheme (acting trivially on $U$). Split gerbes over a real variety seem important in the study of the topology of real Deligne--Mumford stacks in general, and of Conjecture \ref{conj:ST} in particular, as any real Deligne--Mumford stack admits a stratification by real stacks \'etale locally of the form $[U/H]$. 

In the following theorem 
we describe the topology of the real locus of split gerbes.

\begin{theorem} \label{thm:topcovering[U/H]}
Let $U$ be a geometrically connected $\mathbb R$-variety with $U(\RR) \neq \emptyset$. 
Let $H\rightarrow U$ be a finite \'etale group scheme. Let $\sigma_U \colon U(\CC) \to U(\CC)$ and $\sigma_H \colon H(\CC) \to H(\CC)$ be the real structures of the varieties $U$ and $H$ over $\RR$. 
Let $C$ be a connected component of $U(\mathbb R)$, fix $p\in C$, and consider the action of the fundamental group $\pi_1(U(\mathbb C),p)$ on $H_p(\mathbb C)$ attached to the topological covering $H(\CC) \to U(\CC)$. Then the following holds.  
\begin{enumerate}
\item \label{itemthm:UH:2} 
    The image of the natural map $\pi_1(C,p) \to \pi_1(U(\mathbb C),p)$ lies in the subgroup of elements $g \in \pi_1(U(\mathbb C),p)$ whose action on $H_p(\CC)$ is $G$-equivariant. In particular, the group $\pi_1(C,p)$ acts naturally on $\rm  H^1(G,H_p(\CC))$.
    \item \label{itemthm:UH:1} The restriction 
    $f^{-1}(C) \to C$ of the canonical map $f\colon \vert [U/H](\mathbb R)\vert\rightarrow U(\mathbb R)$ is a topological covering with fiber $\rm  H^1(G, H_p(\CC))$, whose corresponding action of $\pi_1(C,p)$ on $\rm  H^1(G, H_p(\CC))$ coincides with the action described in item \ref{itemthm:UH:2} above. 
    \item  \label{item:3:theoremUH}
     Let $D$ be another connected component of $U(\mathbb R)$ and $q\in D$. Choose a topological path $\gamma \colon [0,1] \to U(\mathbb C)$ from $p$ to $q$, and set $\omega \coloneqq (\sigma_U \circ \gamma) \ast \gamma^{-1}\in \pi_1(U(\mathbb C),p)$, where $\ast$ denotes the composition of paths. 
     Let $\phi \colon H_p(\mathbb C)\simeq H_q(\mathbb C)$ be the isomorphism induced by $\gamma$. Then, we have $\sigma_{H}(\phi(h))=\omega \cdot \phi(\sigma_{H}(h))$ for each $h\in H_p(\mathbb C)$. 
\end{enumerate}
\end{theorem}

Item \ref{item:3:theoremUH} of Theorem \ref{thm:topcovering[U/H]} implies that, once the actions of \( G \) on $U(\CC)$ and on the fiber \( H_p(\mathbb{C}) \) above one point \( p \in U(\mathbb{R}) \) are known, one can determine the action of \( G \) on the fiber \( H_q(\mathbb{C}) \) above any other point \( q \in U(\mathbb{R}) \), by looking at the way in which \( G \) acts on a topological path connecting \( p \) and \( q \) in \( U(\mathbb{C}) \). A more general version of this statement, which may be of independent interest, is given in Proposition \ref{prop:changeofbasepoint}. A sample application of this result is given in Example \ref{ex:prop-apply-elliptic}, where we consider the action of $G$ on the cohomology of the fibers of a family of real elliptic curves.  

As an application of Theorem \ref{thm:topcovering[U/H]}, we establish the Smith--Thom inequality for a certain class of split gerbes over a real curve.
\begin{corollary}\label{cor:gerbesovercurves}
    Let $U$ be smooth curve over $\RR$ and $n \in \ZZ_{\geq 0}$. Let $H \to U$ be a finite \'etale abelian group scheme such that $H_x(\CC)[2] = H_x(\CC)[4]$ for $x \in U(\CC)$.  
    Then Conjecture \ref{conj:ST} holds for the stack $[U/H]$. 
\end{corollary}

Moreover, in the proper case, Theorem \ref{thm:topcovering[U/H]} yields the following concrete interpretation of Conjecture \ref{conj:ST}. 

\begin{corollary}\label{cor:gerbesovercurvesconcrete}
 Let $U$ be smooth proper curve of genus $g$ over $\RR$. Fix $p\in U(\mathbb R)$. Let $H \to U$ be a finite \'etale group scheme, and $C_1, \dotsc, C_m$ the connected components of $U(\RR)$. For each $i$, choose a point $p_i \in C$. Consider the natural action of $\pi_1(C_i,p_i)$ on $\rm H^1(G, H_{p_i}(\CC))$, see Theorem \ref{thm:topcovering[U/H]}. Then Conjecture \ref{conj:ST} holds for $[U/H]$ if and only if 
       \begin{align} \label{align:inequality:H/U}
        \sum_{i = 1}^m\# \left(\frac{\rm H^1(G,H_{p_i}(\mathbb C))}{\pi_1(C_i,p_i)} \right) \leq 2 \cdot \# \left(\frac{H_{p}(\mathbb C)}{\pi_1(U(\mathbb C),p)}\right)+ (g-1) \cdot \#H_p(\CC).
    \end{align}
\end{corollary}
By Corollary \ref{cor:gerbesovercurves}, inequality \eqref{align:inequality:H/U} holds when the fibers of $H \to U$ are of the form $(\ZZ/2)^n$ for some $n \geq 0$. We do not know a direct proof of this inequality without passing through the geometry of $[U/H]$, i.e., without using Corollary \ref{cor:gerbesovercurves}. 

Finally, in Section \ref{sec:splitting}, we reformulate the abstract statement of Theorem \ref{thm:topcovering[U/H]} in more concrete group-theoretic terms; see Proposition \ref{prop:splitting}. This  allows one to compute the topological covering appearing in Theorem \ref{thm:topcovering[U/H]} in concrete examples. To illustrate the techniques, we compute the topology of various gerbes over an Enriques surface, and prove the Smith–Thom inequality \eqref{align:inequality-ST-stacks} for such gerbes. 
\subsection{Organization of the paper.}
This paper is organized as follows. In Section \ref{sec:notation}, we fix the conventions and notation used throughout the paper. In Section \ref{sec:topologyofarealquotientstack}, we prove a formula for the real locus of a quotient stack, which we then apply to verify Conjecture \ref{conj:ST} in numerous examples. In Section \ref{sec:STforcurves}, we prove the conjecture for a broad class of real stacky curves. Section \ref{sec:classifyingstack} focuses on the topology of split gerbes, and we use this analysis to establish the conjecture for various split gerbes over a real curve. Finally, in Section \ref{sec:splitting}, we reformulate the results of Section \ref{sec:classifyingstack} in group-theoretic terms, and then use this to prove Conjecture \ref{conj:ST} for several split gerbes over an Enriques surface.

\subsection{Acknowledgements.} We thank Olivier Benoist, Ilia Itenberg, Matilde Manzaroli, Ieke Moerdijk and Florent Schaffhauser for helpful discussions. Special thanks to Matilde Manzaroli for helping us with the pictures in the paper. This research was partly supported by the grant ANR–23–CE40–0011 of Agence National de la Recherche.
The second author has received funding 
 from the European Research Council (ERC) under the European Union’s Horizon 2020 research and innovation programme under grant agreement N\textsuperscript{\underline{o}}948066 (ERC-StG RationAlgic), and from the ERC Consolidator Grant FourSurf N\textsuperscript{\underline{o}}101087365.

\section{Notation and preliminaires} \label{sec:notation}
\subsection{Sets and topological spaces.} If $X$ is a finite set, we write $\#X$ for the cardinality of $X$. 
All the topological spaces in this paper are assumed to be locally compact and Hausdorff.  If $X$ is a topological space such that $\dim \rm H^\ast(X,\ZZ/2)$ is finite, we define $h^\ast(X) = \dim \rm H^\ast(X,\ZZ/2)$. 

A \emph{circle} is a topological space homeomorphic to $\mathbb S^1 \coloneqq \set{z \in \CC \mid \va{z} =1}$. An \emph{interval} is a topological space homeomorphic to one of intervals $(0,1)$, $[0,1)$ or $[0,1]$ inside $\mathbb R$. If the topological space is homeomorphic to $(0,1)$ we call it \emph{open interval}. 
\subsection{Groups.} Throughout the paper, $G$ will be the finite group
$
G = \Gal(\CC/\RR) \simeq \ZZ/2,
$ and $\sigma \in G$ denotes a generator. Let $\Gamma$ be a group on which $G$-acts. Recall (see e.g. \cite[Chapitre I, \S 5]{serre-galoisienne}) that the non-abelian Galois cohomology group 
$\rm  H^1(G, \Gamma)$ can be canonically identified with $\rm Z^1(G,\Gamma)/\sim $ where $\rm Z^1(G,\Gamma)\subset \Gamma$ is the set of $\gamma\in \Gamma$ such that $\gamma\sigma(\gamma)=e$ and $\sim$ is the equivalence relation $\gamma \sim \beta \gamma \sigma(\beta)^{-1}$ for $\beta \in \Gamma(\CC)$. 

If $X$ is a set on which a group $H$ acts, we write $X/H$ for the set of orbits. Unless stated otherwise, we consider $H$ acting on itself by conjugation, so that $H/H$ is the set of conjugacy classes of $H$. 

\subsection{Real varieties.} A \emph{variety} over $\RR$ (resp.\ $\CC$) will be a reduced and separated scheme of finite type over $\RR$ (resp.\ $\CC$).  A variety over $\RR$ will also be called a \emph{real variety}.  If $X$ is a real variety, we denote by  $\sigma_X\colon X(\mathbb C)\rightarrow X(\mathbb C)$  the induced anti-holomorphic involution, which we also view as an action of $G$ on $X(\CC)$.

A \emph{curve} over $\RR$ (resp.\ $\CC$) is a variety of dimension one over $\RR$ (resp.\ $\CC$). Note that we do not assume that $X$ is proper. For a smooth curve $X$ over $\RR$, any connected component $C \subset X(\RR)$ is homeomorphic to either the circle or the open interval. 

    For $n \in \ZZ_{\geq 1}$, we let $\mu_n$ be the $\RR$-group scheme with $\mu_n(S) = \set{x \in \OO_S(S) \mid x^n = 1}$ for a scheme $S$ over $\RR$. 
    
\subsection{Algebraic stacks.}\label{sec:notationstack}
We indicate an algebraic stack by a calligraphic letter, such as $\ca X, \ca Y, \ca Z$. Schemes are usually indicated by roman capitals, such as $X,Y,Z$. For an algebraic stack $\ca X$, we let $\ca I_{\ca X}\to \ca X$ denote the inertia stack over $\ca X$. 


When $\ca X$ is an algebraic stack over a scheme $S$, we let $\va{\ca X(S)}$ denote the set of isomorphism classes of the groupoid $\ca X(S)$. 
For an algebraic stack $\ca X$ over $\RR$, and an object $x \in \ca X(\RR)$, let $x_\CC \in \ca X(\CC)$ denote the pull-back of $x$ along $\Spec(\CC) \to \Spec(\RR)$. 

Let $\mathcal X$ be a Deligne-Mumford stack over $\mathbb R$. Recall that the sets $\vert \mathcal X(\mathbb R)\vert$ and $\vert \mathcal X(\mathbb C)\vert$ are endowed with a natural topology. To define it, one chooses a scheme $U$ and an \'etale morphism $U\rightarrow \mathcal X$ such that 
$U(\mathbb R)\rightarrow \vert \mathcal X(\mathbb R)\vert$ and $U(\mathbb C)\rightarrow \vert \mathcal X(\mathbb C)\vert $ are surjective. Then the topologies on $\vert \mathcal X(\mathbb R)\vert $ and $\vert \mathcal X(\mathbb C)\vert $ are the induced quotient topologies. See \cite[Section 2.3]{thesis-degaayfortman} for more details. 
\subsection{Preliminaries from \cite{AdGF-grup}.} For the convenience of the reader, we recall the following results from \cite{AdGF-grup}, that will be used several times in the sequel. 
\begin{proposition} \label{prop:fiber-H1}\cite[Proposition 5.5]{AdGF-grup}. Let $\ca X$ be a separated Deligne--Mumford stack of finite type over $\RR$, with coarse moduli space $p \colon \ca X \to M$.
    \begin{enumerate}
        \item Let $f \colon \vXR \to M(\RR)$ denote the map induced by $p$. Let $x \in \ca X(\RR)$ with isomorphism class $[x] \in \va{\ca X(\RR)}$. There is a canonical bijection $f^{-1}(f([x])) = \rm  H^1(G, \Aut(x_\CC))$. 
        \item We have $\# \rm  H^1(G, \Aut(x_\CC)) = \# \rm  H^1(G, \Aut(x'_\CC))$ for each pair of objects $x, x' \in \ca X(\RR)$ whose induced objects $x_\CC, x'_\CC \in \ca X(\CC)$ are isomorphic in $\ca X(\CC)$. 
    \end{enumerate}    

\end{proposition}
\begin{proposition}\label{prop:inertia-description}\cite[Proposition 5.4]{AdGF-grup}. Let $S$ be a complex variety. Let $f\colon X \to S$ be a scheme of finite type over $S$. Let $H \to S$ be a finite group scheme over $S$, acting  on $X$ over $S$. Then
    \begin{align} 
\va{\ca I_{[X/H]}(\CC)}= \set{(x,\gamma) \in X(\CC) \times H(\CC) \mid \gamma \in \Stab_{H_{f(x)}(\CC)}(x)} /_{\sim}
\end{align}
where $(x,\gamma) \sim (gx, g\gamma g^{-1})$ for $g \in H_{f(x)}(\CC)$. 
\end{proposition}
\section{Topology of a real quotient stack}\label{sec:topologyofarealquotientstack}
In this section, we describe the topology of the real locus of the quotient stack $[X/\Gamma]$ of a real variety $X$ on which a finite $\mathbb R$-group scheme $\Gamma$ acts, and prove Theorem \ref{thm:introlemme-prefere}. We then use this description to verify the Smith--Thom inequality \eqref{align:inequality-ST-stacks} in many examples. Finally, we use Theorem \ref{thm:introlemme-prefere} to study the topology of the coarse moduli space map $\ca X \to M$ of a separated Deligne--Mumford stack $\ca X$ over $\RR$, see Theorem \ref{thm:coarse-map-closed}.. 

\subsection{The real locus of a quotient stack.}\label{sec:real-locus-quotient} The goal of Section \ref{sec:real-locus-quotient} is to prove Theorem \ref{thm:introlemme-prefere}. 
\subsubsection{Group schemes over the reals and torsors.}\label{}Let $\Gamma$ be a finite group scheme over $\RR$ and let $\sigma_{\Gamma} \colon \Gamma(\CC)\rightarrow \Gamma(\CC)$
be the action of $G$ on $\Gamma(\CC)$ corresponding to $\Gamma$. 

Choose a complete set of representative $H\subset \rm Z^1(G, \Gamma) $ for $H^1(G,\Gamma)$, so that the composition $H \subset \rm Z^1(G,\Gamma) \to \rm  H^1(G,\Gamma)$ is a bijection; we choose $H$ such that $e \in H$. For each $\gamma\in H$, we define an involution
\[
\varphi^{\gamma} \colon \Gamma(\CC)\rightarrow \Gamma(\CC) \quad \quad \text{as} \quad \quad \varphi^{\gamma}(g)= \sigma(g) \cdot \gamma^{-1}.
\]
We consider the resulting $G$-set $(\Gamma(\CC), \varphi^\gamma)$. Note that left multiplication defines an action of the $G$-group $(\Gamma(\CC), \sigma_\Gamma)$ on the $G$-set $(\Gamma(\CC), \varphi^\gamma)$. In particular, if $P_\gamma$ is the $\RR$-scheme associated to $(\Gamma(\CC), \varphi^\gamma)$, 
we get an action of the $\RR$-group scheme $\Gamma$ on the $\RR$-scheme $P_\gamma$ that turns the latter into a $\Gamma$-torsor. 
Recall the following classical result. 
\begin{lemma}\label{lem:bijectiontorsor}
The following map is bijective:
	\begin{align*} \rm  H^1(G, \Gamma) = H &\rightarrow \{\text{isomorphism classes of } \Gamma \text{-torsors over } \Spec(\mathbb R)\}, \\
    \gamma&\mapsto P_\gamma.
    \end{align*}
\end{lemma}
\begin{proof}
See e.g.~\cite[Example 2, page 13]{MR1845760}. 
\end{proof}
\subsubsection{The topology of the real locus of a quotient stack.} \label{sec : twistingdefinition}
We continue with the above notation. For $\gamma \in H$, define an involution
\[
\sigma^{\gamma}_{\Gamma}\colon \Gamma(\CC)\rightarrow \Gamma(\CC) \quad \quad \text{as} \quad \quad \sigma^{\gamma}_{\Gamma}(g) \coloneqq \gamma\sigma_{\Gamma}(g)\gamma^{-1}.
\]
The pair $(\Gamma(\CC), \sigma_\Gamma^\gamma)$ corresponds to a finite group scheme $\Gamma_\gamma$ over $\RR$.
Let $X$ be a quasi-projective scheme over $\RR$ with real structure $\sigma_X \colon X(\CC) \to X(\CC)$, acted upon from the left by the finite group scheme $\Gamma$ over $\RR$. 
    For $\gamma \in H$, 
    define an involution $\sigma^\gamma_X \colon X(\CC) \to X(\CC)$ as  $\sigma^\gamma_X(x) = \gamma \cdot \sigma_X(x)$. 
    The pair $(X(\CC), \sigma_X^\gamma)$ corresponds to a quasi-projective scheme $X_\gamma$ over $\RR$. 
     Note that 
    \[
    X_\gamma(\RR) = X(\CC)^{\sigma_X^\gamma} \quad \text{and} \quad \Gamma_\gamma(\RR) = \Gamma(\CC)^{\sigma_\Gamma^\gamma} \quad \text{for each} \quad \gamma \in H.
    \]
\begin{proof}[Proof of Theorem \ref{thm:introlemme-prefere}]
Recall that, by definition, one has
$$ \va{[X/\Gamma](\RR)}=\set{\text{pairs } (P, f) \mid P \text{ a $\Gamma$-torsor and } f \text{ a $\Gamma$-equivariant morphism } P \to X }/_{\simeq}$$
where $(P_1,f_1)\simeq (P_2,f_2)$ if there exists an isomorphism of torsors $h\colon P_1\xrightarrow{\simeq} P_2$ such that $f_2\circ h=f_1$. 

We need to prove that there exists a canonical homeomorphism
    \[
    \va{[X/\Gamma](\RR)} \xrightarrow{\sim} \coprod_{\gamma \in H} X_\gamma(\RR)/\Gamma_\gamma(\RR).
    \]
To prove this, we first observe that the action of $\Gamma(\CC)$ on $X(\mathbb C)$ is compatible with the involutions $\sigma^{\gamma}_{X}$ and $\sigma^{\gamma}_{\Gamma}$. Indeed, for $x \in X(\CC)$ and $g \in \Gamma(\CC)$, we have:
\[
\sigma^\gamma_X(g \cdot x) = \gamma \cdot \sigma_X(g \cdot x) = \gamma\cdot \sigma_\Gamma(g) \cdot \sigma_X(x) = \gamma \cdot \sigma_\Gamma(g) \cdot \gamma^{-1} \cdot \gamma \cdot \sigma_X(x) = \sigma_\Gamma^\gamma(g) \cdot \sigma_X^\gamma(x).
\]
Therefore, we obtain an action of the $G$-group $(\Gamma(\CC), {\sigma_\Gamma^\gamma})$ on the $G$-space $(X(\CC), \sigma_X^\gamma)$. In particular, the subgroup 
\[
\Gamma_\gamma(\RR) = \Gamma(\CC)^{\sigma^\gamma_\Gamma} \subset \Gamma(\CC)
\]
of elements fixed under $\sigma^\gamma_\Gamma$ acts on the fixed space $X_\gamma(\RR) = X(\CC)^{\sigma_X^\gamma} \subset X(\CC)$. 

Fix $\gamma\in H$ and take any $x\in X_\gamma(\RR)/\Gamma_\gamma(\RR)$. Choose a $y\in X_\gamma(\mathbb R)$ that lifts $x$ and consider the $\Gamma(\CC)$-equivariant morphism
$f_y \colon \Gamma(\CC)\rightarrow X(\mathbb C)$ defined as $g \mapsto g \cdot y.$ This morphism is compatible with the $G$-action $\varphi^{\gamma}$ on $\Gamma(\CC)$ and with the $G$-action $\sigma_{X}$ on $X(\CC)$, hence it gives rise to a $\Gamma$-equivariant morphism 
$f_y \colon P_\gamma\rightarrow X$
of schemes over $\RR$. Define 
$$\alpha(x) \coloneqq (P_\gamma,f_y)\in \va{[X/\Gamma](\mathbb R)}.$$ 
We first show that $\alpha$ is well-defined, i.e., that $\alpha$ does not depend on the choice of the lift $y$ of $x$. If $z\in X_\gamma(\RR)$ is another lift of $x$, then there exists a $g\in  \Gamma_\gamma$ such that $y=g \cdot z$. Since $g\in \Gamma_\gamma(\RR) = \Gamma(\CC)^{\sigma_\Gamma^\gamma}$, the morphism $g\colon P_\gamma\rightarrow P_\gamma$ sending $h$ to $hg$ is  an isomorphism of torsors over $\mathbb R$, fitting into the commutative diagram
\begin{center}
	\begin{tikzcd}
P_\gamma\arrow{r}{f_y}\arrow{d}{g} & X\arrow[equal]{d}\\
	P_\gamma\arrow{r}{f_z} & X. 
			\end{tikzcd}
	\end{center} 
In particular, we have an equality of isomorphism classes $[(P_\gamma,f_y)]=[(P_\gamma,f_z)] \in \va{[X/\Gamma](\mathbb R)}$. We conclude that we get a canonical map
\begin{align}\label{align:alpha-bijection}
\alpha \colon   \va{[X/\Gamma](\RR)} \to \coprod_{\gamma \in H} X_\gamma(\RR)/ \Gamma_\gamma(\RR).
\end{align}
We observe that, by construction and Lemma \ref{lem:bijectiontorsor}, the map $\alpha$ is surjective. It is also injective since $\Aut_{\Gamma}(P_{\gamma})=\Gamma_{\gamma}(\mathbb R)$. Thus, it remains to prove that $\alpha$ is a homeomorphism. 

To see this, note that for each $\gamma \in H$, we have a natural morphism 
\begin{align}\label{align:nat-mor}
X_\gamma \to [X/\Gamma].
\end{align}
Namely, to give such a map is to give:
\begin{enumerate}
    \item[(1)] a $\Gamma \times_\RR {X_{\gamma}}$-torsor $P \to X_\gamma$ over $X_\gamma$, and
    \item[(2)] a $\Gamma \times_\RR {X_{\gamma}}$ equivariant morphism $P \to X \times_\RR X_\gamma$ of schemes over $X_\gamma$.
\end{enumerate}
As for (1), we put $P = P_\gamma \times_\RR X_\gamma$, which is a $\Gamma \times_\RR X_\gamma$-torsor by base-changing the $\Gamma$-torsor structure of $P_\gamma \to \Spec(\RR)$ along $X_\gamma \to \Spec(\RR)$. As for (2), we consider the morphism 
\begin{align} \label{align:mor-Pgamma}
P_\gamma \times_\RR X_\gamma \to X \times_\RR X_\gamma 
\end{align}
defined via Galois descent by the map
\[
\Gamma(\CC) \times X(\CC) \to X(\CC) \times X(\CC), \quad \quad (g,x) \mapsto (gx,x),
\]
which is indeed compatible with the anti-holomorphic involution $(g,x) \mapsto (\varphi^\gamma(g), \sigma_X^\gamma(x))$ on the left hand side and the anti-holomorphic involution $(x,y) \mapsto (\sigma_X(x), \sigma_X^\gamma(y))$ on the right hand side. Since the map \eqref{align:mor-Pgamma} is $\Gamma \times_\RR X_\gamma$-equivariant, it yields the desired morphism \eqref{align:nat-mor}. 

We obtain a morphism
$
U \coloneqq \coprod_{\gamma \in H} X_\gamma \to [X/\Gamma],
$
and, by the fact that the map $\alpha$ in \eqref{align:alpha-bijection} is a bijection (which has already been shown), the induced map
\begin{align} \label{align:induced-surjective}
U(\RR) = \coprod_{\gamma \in H} X_\gamma(\RR) \to \va{[X/\Gamma](\RR)}
\end{align}
is surjective. By definition of the real analytic topology on $\va{[X/\Gamma](\RR)}$, see Section \ref{sec:notationstack}, it follows that the topology on $\va{[X/\Gamma](\RR)}$ is the quotient topology coming from the surjection \eqref{align:induced-surjective} and the real analytic topology on $U(\RR) = \coprod_{\gamma} X_\gamma(\RR)$. As the diagram
\[
\xymatrix{
\coprod_{\gamma \in H} X_\gamma(\RR)\ar@{=}[r] \ar[d] &\coprod_{\gamma \in H} X_\gamma(\RR) \ar[d] \\
\va{[X/\Gamma](\RR)} \ar[r]^-{\alpha} & \coprod_{\gamma \in H} X_\gamma(\RR)/\Gamma_\gamma 
}
\]
commutes, and as each quotient $X_\gamma(\RR)/\Gamma_\gamma $ carries the quotient topology coming from  $X_\gamma(\RR) \to X_\gamma(\RR)/\Gamma_\gamma $, this proves that $\alpha$ is a homeomorphism as wanted.
\end{proof}

\subsection{Smith--Thom for various quotient stacks.} \label{sec:smiththom-application-quotient}In this section we apply Theorem \ref{thm:introlemme-prefere} to prove the Smith--Thom inequality (\ref{align:inequality-ST-stacks}) in a number of examples.
\begin{example}
Let $\Gamma$ be any finite $\mathbb R$-group scheme. 
Take $X=\Spec(\mathbb R)$ with the trivial action of $\Gamma$. Then Theorem \ref{thm:introlemme-prefere} just says that
$\va{[X/\Gamma](\mathbb R)}$ is the disjoint union of  $\# \rm H^1(G,\Gamma)$ points, which also follows directly from the definitions and Lemma \ref{lem:bijectiontorsor}. (Note that we already verified the Smith--Thom inequality (\ref{align:inequality-ST-stacks}) for $[X/\Gamma]$ in \cite[Proposition 1.8]{AdGF-grup}.) 
\end{example}

\begin{example}\label{ex:A1modZ2}
Let $X \coloneqq \mathbb A^1_{\mathbb R}$. Let $\Gamma \coloneqq \ZZ/2$ endowed with the trivial $G$-action. We let $\Gamma$ act on $X$ via the map sending $x$ to $-x$. To compute $\mathcal X \coloneqq [\mathbb A^1_{\mathbb R}/\Gamma]$, we start by observing that $\rm  H^1(G,\Gamma)$ has two elements, $e$ (the trivial element) and $\gamma$. One computes that $X(\mathbb R)/\Gamma = \RR/\pm 1 \simeq \mathbb R_{\geq 0}$ and also $X_{\gamma}(\mathbb R)/\Gamma = i\RR/\pm 1 \simeq i\mathbb R_{\geq 0}$. Hence, by Theorem \ref{thm:introlemme-prefere}, we have
$$\vert \mathcal X(\mathbb R)\vert = \RR_{\geq 0} \coprod i\RR_{\geq 0}.$$
In conclusion, we find that $h^*(\vert \mathcal X(\mathbb R)\vert )=2$. On the other hand, by Proposition \ref{prop:inertia-description}, one has  $I_{[X/\Gamma]}(\mathbb C)\simeq \mathbb C\coprod \{0\}.$
    In particular $h^*(I_{[X/\Gamma]}(\mathbb C))=2$, we see that the Smith--Thom inequality (\ref{align:inequality-ST-stacks}) holds and it is an actual equality. 
For completeness, we also describe the natural map $f\colon \vert [\mathbb A^1_{\mathbb R}/\Gamma](\mathbb R)\vert \rightarrow (X/\Gamma)(\mathbb R)$ (see Figure \ref{fig: A1:Z2}). Identifying  $(X/\Gamma)(\mathbb C)$ with $\mathbb C$ via the map $z\mapsto z^2$, one sees that  $(X/\Gamma) (\mathbb R)=\mathbb R\subset \mathbb C$. Under this identification, $f$ induces homeomorphisms $X(\mathbb R)/\Gamma = \RR_{\geq 0} \xrightarrow{z \mapsto z^2} \mathbb R_{\geq 0}$ and $X(\mathbb R)_{\gamma}/\Gamma = i\RR_{\geq 0} \xrightarrow{z \mapsto z^2}\mathbb R_{\leq 0}$. Hence $\# f^{-1}(x)=1$ for every $x\neq 0$ and $\# f^{-1}(0)=2$, as predicted by Proposition \ref{prop:fiber-H1}.
\begin{figure}[h!]
\begin{center}
\begin{picture}(0,80)
\put(-150,15){$(\mathbb A^1/\ZZ/2)(\mathbb R)$}
\put(-150,65){$\vert [\mathbb A^1/\ZZ/2](\mathbb R)\vert $}
\put(-40,15){\includegraphics[width=0.2\textwidth]{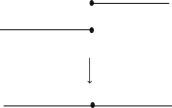}}
\end{picture}
\caption{The morphism $\vert [\mathbb A^1/(\ZZ/2)](\mathbb R)\vert \rightarrow (\mathbb A^1/(\ZZ/2))(\mathbb R)$}
\label{fig: A1:Z2}
\end{center}
\end{figure}
\end{example}

\begin{example}\label{ex:A1modZ2xZ2} 
Let $X \coloneqq \mathbb A^1_{\mathbb R}$. 
    Let $\Gamma \coloneqq \ZZ/2\times \ZZ/2$ endowed with the $G$-action exchanging the coordinates. We let $\Gamma$ act on $\mathbb A^1_{\mathbb R}$ via $(a,b)*x \coloneqq (-1)^{a+b}x$. To compute $\vert \mathcal X(\mathbb R)\vert  \coloneqq [\mathbb A^1_{\mathbb R}/\Gamma]$, we start observing that $\rm  H^1(G,\Gamma)=0$. Hence 
$\vert \mathcal X(\mathbb R)\vert =X(\mathbb R)/\Gamma(\mathbb R)= X(\mathbb R)=\mathbb R$
since $\Gamma(\mathbb R)$  acts trivially on $X(\mathbb C)$. We find that $h^*(\vert \mathcal X(\mathbb R)\vert) =1$. On the other hand, by Proposition \ref{prop:inertia-description}, one has 
 $I_{[X/\Gamma]}(\mathbb C)\simeq \mathbb A^1\coprod \mathbb A^1 \coprod \{0\}\coprod \{0\}$ so that $h^*(I_{\mathcal X(\mathbb C)})=4$. Hence the Smith--Thom inequality (\ref{align:inequality-ST-stacks}) holds and it is a strict inequality. 
 
 For completeness, we describe the natural map $f\colon \vert [\mathbb A^1_{\mathbb R}/\Gamma](\mathbb R)\vert \rightarrow (X/\Gamma)(\mathbb R)$ (see Figure \ref{fig: A1:Z2xZ2}). As in Example \ref{ex:A1modZ2}, one identifies $(X/\Gamma)(\mathbb R)$ with $\mathbb R\subset \mathbb C$. Under this identification, the map $f\colon \mathbb R\rightarrow \mathbb R$ becomes the absolute value map (hence not surjective), so that $\# f^{-1}(x)=2$ for every $x>0$, and $\# f^{-1}(0)=1$, as predicted by Proposition \ref{prop:fiber-H1}.
\begin{figure}[h!]
\begin{center}
\begin{picture}(0,80)
\put(-150,15){$(\mathbb A^1/(\ZZ/2))(\mathbb R)$}
\put(-150,55){$\vert [\mathbb A^1/(\ZZ/2)](\mathbb R)\vert $}
\put(-40,15){\includegraphics[width=0.2\textwidth]{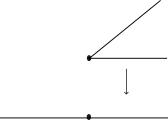}}
\end{picture}
\caption{The morphism $\vert [\mathbb A^1/(\ZZ/2\times \ZZ/2)](\mathbb R)\vert \rightarrow (\mathbb A^1/(\ZZ/2\times\ZZ/2))(\mathbb R)$}
\label{fig: A1:Z2xZ2}
\end{center}
\end{figure}
\end{example}
\begin{examples}
Let $X\coloneqq \mathbb A_{\mathbb R}^2$.
\begin{enumerate}
    \item Let  $\Gamma \coloneqq \ZZ/2$ endowed with the trivial $G$-action. We let $\Gamma$ act on $X$ via the map sending $x$ to $-x$. To compute the real locus of $\mathcal X \coloneqq [\mathbb A^2_{\mathbb R}/\Gamma]$, we start by observing that $\rm  H^1(G,\Gamma)$ has two elements, $e$ (the neutral element) and $\gamma$. By Theorem \ref{thm:introlemme-prefere}, 
$$\vert \mathcal X(\mathbb R)\vert =X(\mathbb R)/\Gamma(\mathbb R)\coprod X_{\gamma}(\mathbb R)/\Gamma(\mathbb R).$$
One computes that $X(\mathbb R)/\Gamma$ and $X_{\gamma}(\mathbb R)$ are two half-planes, so that $h^*(\vert \mathcal X(\mathbb R))\vert =2$. By Proposition \ref{prop:inertia-description}, one has  $\vert \ca I_{[X/\Gamma]}(\mathbb C)\vert\simeq \mathbb A^2\coprod \mathbb A^1$ so that $h^*(\vert \ca I_{[X/\Gamma]\vert}(\mathbb C))=2$. Hence the Smith--Thom inequality (\ref{align:inequality-ST-stacks}) holds, and is an equality.
\item Let $\Gamma \coloneqq \ZZ/2\times \ZZ/2$ endowed with the $G$-action exchanging the coordinates. We let $\Gamma$ act on $\mathbb A^2_{\mathbb R}$ via its $G$-equivariant quotient $\mathbb Z/2$, acting by exchange of coordinates. To compute $\mathcal X \coloneqq [\mathbb A^2_{\mathbb R}/\Gamma]$, we start by observing that $\rm  H^1(G,\Gamma)=0$, so that, by Theorem \ref{thm:introlemme-prefere}, 
$$\vert \mathcal X(\mathbb R)\vert =X(\mathbb R)/\Gamma(\mathbb R)=\mathbb R^2$$
since  $\Gamma(\mathbb R)$ acts trivially on $X(\mathbb C)$. We find that $h^*(\vert \mathcal X(\mathbb R)\vert =1$. 
By Proposition \ref{prop:inertia-description}, we have $I_{[X/\Gamma]}(\mathbb C)\simeq \mathbb C^2\coprod \mathbb C^2\coprod \mathbb C^1\coprod \mathbb C^1$. In particular, $h^*(I_{[X/\Gamma]}(\mathbb C))=4$. Hence the Smith--Thom inequality (\ref{align:inequality-ST-stacks}) holds, and is a strict inequality.
\end{enumerate}
\end{examples}
\begin{example}\label{ex:symmetricproduct}
    Let $Y$ be a real algebraic variety, let $\Gamma \coloneqq \ZZ/2$ act on $Y \times_\RR Y$ by exchanging the coordinates and let $\mathcal X \coloneqq [(Y \times_\RR Y)/\Gamma]$. If $\gamma$ is the non-trivial element of $\rm  H^1(G,\Gamma)$, by Theorem \ref{thm:introlemme-prefere}, one has
    $$\vert \mathcal X(\mathbb R)\vert \simeq (Y(\mathbb R)\times Y(\mathbb R))/\Gamma \coprod (Y \times_\RR Y)_{\gamma}(\mathbb R)/\Gamma\simeq (Y(\mathbb R)\times Y(\mathbb R))/\Gamma \coprod Y(\mathbb C)/G.$$
    Observe that 
    $$\va{\mathcal X(\mathbb C)}^{G}\simeq (Y(\mathbb R)\times Y(\mathbb R))/\Gamma \coprod_{Y(\mathbb R)} Y(\mathbb C)/G,$$
    where $i:Y(\mathbb R)\hookrightarrow \va{\mathcal X(\mathbb C)}^{G}$ embeds diagonally in $Y(\mathbb R)\times Y(\mathbb R)$ and naturally in $Y(\mathbb C)/G$. 
 If $f\colon\vert \mathcal X(\mathbb R)\vert \rightarrow (Y \times_\RR Y/\Gamma)(\mathbb C)^{G}$ is the natural morphism, the exact sequence of sheaves
 $0\rightarrow \ZZ/2\rightarrow f_*\ZZ/2\rightarrow i_*\ZZ/2\rightarrow 0$  shows that 
 \begin{equation}\label{eq:STsymmetric}
          h^*(\vert \mathcal X(\mathbb R)\vert )\leq h^*(\va{\mathcal X(\mathbb C)}^G)+h^*(Y(\mathbb R)).
 \end{equation}
  On the other hand, by Proposition \ref{prop:inertia-description}, we have
  $\va{\ca I_{\ca X}(\mathbb C)}\simeq \va{\mathcal X(\mathbb C)}\coprod Y(\mathbb C),$
   so that by the classical Smith--Thom inequality for $\va{\mathcal X(\mathbb C)}\coprod Y(\mathbb C)$, we get 
        $$h^*(\va{\mathcal X(\mathbb C)}^G)+h^*(Y(\mathbb R))\leq h^*(\va{\mathcal X(\mathbb C)})+h^*(Y(\mathbb C))=h^*(\va{\ca I_{\ca X}(\mathbb C)}).$$
        Combining this with (\ref{eq:STsymmetric}), we get that the Smith--Thom inequality (\ref{align:inequality-ST-stacks}) is satisfied. 
\end{example}
\begin{example}\label{ex:exampleabelianvarieties}
    Let $A$ be a real abelian variety of dimension $g$, so that $A(\mathbb R)\simeq (\mathbb S^1)^{g}\times (\ZZ/2)^k$ for some $0\leq k\leq g$ compatibly with the group structure. Consider the inversion $[-1]:A\rightarrow A$. Let $\mathcal X \coloneqq [A/(\ZZ/2)]$ where $\ZZ/2$ acts via $[-1]$. Since the set of points of $\mathcal X$ with non-trivial automorphisms is discrete, the Smith-Thom inequality for $\mathcal X$ follows from Theorem \ref{thm : discrete} that we prove in Section \ref{sec:discrete}. We use Theorem \ref{thm:introlemme-prefere} to describe $\vert \mathcal X(\mathbb R)\vert$. Let $\gamma$ be non-trivial element in $H^1(G, \ZZ/2)$ and observe that, by construction, $A_{\gamma}$ is the quadratic twist of $A$, hence $A(\mathbb R)\simeq A_{\gamma}(\mathbb R)$ as topological groups. Hence, by Theorem \ref{thm:introlemme-prefere}, we have
   $\vert \mathcal X(\mathbb R)\vert  = A(\mathbb R)/[-1]\coprod A(\mathbb R)/[-1],$
while, by Proposition \ref{prop:inertia-description}, 
$$\vert \ca I_{\mathcal X}(\mathbb C)\vert\simeq A(\mathbb C)/[-1]\coprod\left(\coprod_{x\in A(\mathbb C)[2]} \{x\} \right).$$
     In general, the topology of the spaces $A(\mathbb R)/[-1]$ and $A(\mathbb C)/[-1]$ might be  complicated; for small $g$, one has 
     $$
     A(\mathbb R)/[-1] =
\begin{cases}
[0,1] \times (\mathbb Z/2)^{k}, & \text{if } g=1 \\ 
S^2\times (\mathbb Z/2)^{k}, & \text{if } g=2, 
\end{cases}\qquad \text{and}\qquad A(\mathbb C)/[-1]=\begin{cases}
\mathbb P^1(\mathbb C), & \text{if } g=1 \\ 
\mathrm{Kum}(A), & \text{if } g=2,
\end{cases}$$
where $\mathrm{Kum}(A)$ is the singular Kummer surface of $A$. 
Hence, we get
$$
h^*(\vert \mathcal X(\mathbb R)\vert) =\begin{cases}
2k, & \text{if } g=1 \\ 
4k, & \text{if } g=2, 
\end{cases}\qquad \text{and}\qquad h^*(\vert \ca I_{\mathcal X}(\mathbb C)\vert)=\begin{cases}
6, & \text{if } g=1 \\ 
34 & \text{if } g=2,
\end{cases}$$
where the equality $h^*(\vert \ca I_{\mathcal X}(\mathbb C)\vert)=34$ for $g=2$ follows from \cite{mathkummercomputation}. This verifies the inequality \eqref{align:inequality-ST-stacks} once again for $g=1,2$, since $k\leq g$.
\end{example}
\subsection{The topology of the coarse moduli space map.} The goal of this section is 
to prove Theorem \ref{thm:coarse-map-closed}. 
For this, we need the following lemma which seems interesting in its own right, and whose proof follows roughly from \cite[Lemma 2.2.3]{abramvistoli-2002} and Galois descent. 
\begin{lemma}\label{lemma:local-structure}
Let $\ca X$ be a separated Deligne--Mumford stack of finite type over $\RR$, with coarse moduli space $\ca X \to M$. Let $x \in M(\RR)$. Then there is a scheme $N$ over $\RR$ with a real point $y \in N(\RR)$, an \'etale morphism $g \colon N \to M$ with $g(y) = x$, an affine scheme $Y$ of finite type over $\RR$, and a finite group scheme $H$ over $\RR$ that acts on $Y$ over $\RR$, such that $\ca X \times_M N \simeq [Y/H]$. 
\end{lemma}
\begin{proof}
Let $M^{h}$ (resp.\ $M^{sh}$) be the spectrum of the henselization (resp.\ strict henselization) of $M$ at the point $x \in M$. Let $V \to \ca X$ be a finite type \'etale morphism over $\RR$, such that $V(\CC) \to \va{\ca X(\CC)} \to M(\CC)$ has $x$ in its image, and lift $x$ to a point $z \in V(\CC)$. Let $p \in M^h$ be the closed point of $M^h$, let $u  = (z,p) \in (V \times_M M^h)(\CC)$, and let $U$ be the connected component of $V \times_M M^{h}$ that contains $u$. 

Our first goal is show that there exists a finite $\RR$-group scheme $H$ acting on $U$ over $\RR$, such that $\ca X^h = [U/H]$. Our second goal is to show that this implies that there exists an étale neighbourhood $M_i \to M$ of $x \in M$, with $M_i$ affine and of finite type over $\RR$ and equipped with a point $y \in M_i(\RR)$ mapping to $x$, as well as an affine $\RR$-scheme $U_i$ which is finite over $M_i$ and acted upon by $H$ over $\RR$, such that $\ca X \times_M M_i = [U_i/H]$. By taking $N = M_i$ and $Y = U_i$, this will prove the lemma. 

Consider the natural morphism $\pi \colon U \to M^h$. Since $U$ is connected, $M^h$ the spectrum of a henselian local ring, $\pi(u) = p$, and $\pi$ quasi-finite, it follows that $\pi$ is finite (cf.\ \cite[EGA IV, Théorème (18.5.11)]{EGA}). Define $\ca X^{sh} = \ca X \times_M M^{sh}$, $U' = U \times_{M^h}M^{sh}$ and $R' = U' \times_{\ca X^{sh}} U'$. Then $R' \to U'$ is finite étale and $U'$ is the spectrum of a strictly henselian local ring, so we have a canonical isomorphism $R' = U' \times \Gamma$, where $\Gamma$ is the set of connected components of $R'$. Moreover, as in the proof of \cite[Lemma 2.2.3]{abramvistoli-2002}, the groupoid scheme structure $[R'\rightrightarrows U']$ yield a group structure on $\Gamma$ and an action of $\Gamma$ on $U'$ such that $\ca X^{sh} = [U'/\Gamma']$. 

Define $R = U \times_{\ca X^h} U$. The next step is to show that $R = U \times_\RR H$ for a finite $\RR$-group scheme $H$ that acts on $U$ over $\RR$, and that $\ca X^h = [U/H]$. To prove this, consider the finite group $\Gamma = \pi_0(R')$ defined above. 
We have $R' = R \times_{M^h}M^{sh}$  and since the map $M^{sh} \to M^h$ is a Galois covering with Galois group $G\simeq \ZZ/2$, there is a canonical action of $G$ on $R'$ over the $G$-action on $M^{sh}$, compatible with the natural $G$-action on $U' = U \times_{M^h}M^{sh}$. Via the canonical isomorphism $R' = U' \times \Gamma$, we get a $G$-action on $\Gamma$, hence $\Gamma$ descends to a finite $\RR$-scheme $H$ such that there is a canonical isomorphism $R = U \times_\RR H$ that base changes to the canonical isomorphism $R' = U' \times \Gamma$. The composition $R \times_U R \to R$ of the groupoid scheme $[R \rightrightarrows U]$ yield a composition morphism $H \times_\RR H \to H$; since the identity $U \to R = U \times_\RR H$ base changes to the identity $U' \to R' = U' \times \Gamma$, we see that the neutral element $e \in \Gamma$ is $G$-invariant. In this way, $H \to \Spec(\RR)$ acquires the structure of a finite group scheme, and the second projection $R \to U$ yields an action of $H$ on $U$ over $\RR$ such that $\ca X^h = [U/H]$. 


To finish the proof, let $x \in W \subset M$ be an affine open neighbourhood of $x$ and $\set{M_i}_{i \in I}$ an inverse system of affine schemes $M_i$ such that each $M_i$ has a real point $x_i \in M_i(\RR)$ mapping to $x \in M(\RR)$ under an étale map $M_i \to W$, and such that $\varprojlim M_i = M^h$. Define $H_{M^h} = H \times_\RR M^h$ and $H_i = H \times_\RR M_i$ for $i \in I$. In view of \cite[\href{https://stacks.math.columbia.edu/tag/01ZM}{Tag 01ZM}]{stacks-project}, the groupoid scheme $[R \rightrightarrows U]$ over $M^{sh}$ descends to a groupoid scheme $[R_i \rightrightarrows U_i]$ over $M_i$ for some $i$, such that the following holds: we have $R_i = U_i \times_{M_i} H_i = U_i \times_\RR H$, the second projection $R_i \to U_i$ corresponds to an action 
$U_i \times_{M^i} H_i = U_i \times_\RR H \to U_i$ of the $\RR$-group scheme $H$ on the scheme $U_i$ over $\RR$, and $\ca X \times_{M}M_i = [U_i/H]$. The scheme $M_i$ is affine, and of finite type over $\RR$ because $M_i \to W$ is of finite type. As the map $U_i \to M_i$ is finite, $U_i$ is an affine scheme of finite type over $\RR$, and we are done. 
\end{proof}

\begin{proof}[Proof of Theorem \ref{thm:coarse-map-closed}]
   Let $\ca X$ be a separated Deligne--Mumford stack of finite type over $\RR$, with coarse moduli space $\ca X \to M$. By Lemma \ref{lemma:local-structure}, there exist quasi-projective schemes $Y_i$ over $\RR$, finite group schemes $\Gamma_i$ over $\RR$ acting on $Y_i$ over $\mathbb R$, and \'etale morphisms $U_i\rightarrow \mathcal X$ such that $\coprod_i U_i(\mathbb R)\rightarrow M(\mathbb R)$ is surjective and $\mathcal X\times_{M} U_i\simeq [Y_i/\Gamma_i]$. By  \cite[Lemma 2.2.2]{abramvistoli-2002}, the morphism $\mathcal X \times_M U_i \to U_i$ is the coarse moduli space of the stack $\ca X \times_M U_i$, so that we have $U_i = Y_i/\Gamma_i$. Since the diagram
   \begin{center}
	\begin{tikzcd}
   \coprod_i\va{ \left(\mathcal X\times_{M} U_i\right)(\mathbb R)} \arrow{r}\arrow{d} &\coprod_i U_i(\mathbb R) \arrow[two heads]{d} \\
 \vert \mathcal X(\mathbb R)\vert\arrow{r} & M(\mathbb R)
			\end{tikzcd}
	\end{center} 
    is cartesian, and each $U_i(\mathbb R)\rightarrow M(\mathbb R)$ is a local homeomorphism by \cite[Lemma 5.10]{AdGF-grup}, it is enough to show that $[Y_i/\Gamma_i](\mathbb R)\rightarrow (Y_i/\Gamma_i)(\mathbb R)$ is closed. This is a consequence of Theorem \ref{thm:introlemme-prefere}, as for any scheme $Y$ of finite type over $\RR$, and any finite $\RR$-group scheme $H$ acting on $Y$ over $\RR$, the natural map $Y(\RR)/H(\RR) \to (Y/H)(\RR)$ is closed. 
\end{proof}
   \subsection{Smith--Thom for stacks with discrete non-trivial stabilizer locus.}\label{sec:discrete}
   
   \begin{proof}[Proof of Theorem \ref{thm : discrete}]
       Retain the notation and the assumption as in Theorem \ref{thm : discrete}.
Let $\ca U \subset \mathcal X$ be the open substack whose points have trivial stabilizer group; we have $\ca U \simeq U$ for an algebraic space $U$ with an open immersion $U \hookrightarrow \ca X$. Write $\vert \mathcal X(\mathbb C)\vert -U(\CC)=\{p_1,\dots, p_m\}$ for some $p_i \in \va{\ca X(\CC)}$. Choose $x_j\in \mathcal X(\mathbb C)$ with isomorphism class $p_j \in \va{\ca X(\CC)}$. Let $K$ be the image of $f \colon \vert \mathcal X(\mathbb R)\vert \rightarrow \vert \mathcal X(\mathbb C)\vert^{G}$. Up to relabeling, we can assume that $p_1,\dots,p_n$ are in $K$ and $p_{n+1}\dots p_m$ are not.

 Let $V = U(\CC) \cap K \subset \va{\ca X(\CC)}$. 
 Abusing notation, we write $ f \colon \vert \mathcal X(\mathbb R)\vert \rightarrow K$ for the induced map. We denote by $g \colon V\hookrightarrow K$  the inclusion map. 
 For $j \in \set{1, \dotsc, n}$, let $i_j \colon \{p_j\}\hookrightarrow K$ be the inclusion  of the point $p_j$, and consider the following commutative diagram sheaves on $K$ whose rows are exact: 
 
\begin{equation}\label{align:diagram}
\begin{tikzcd}
0\arrow{r} &g_{!}\mathbb Z/2\arrow{r}\arrow{d} & \mathbb Z/2\arrow{r}\arrow{d} & \underset{1\leq j\leq n}{\bigoplus}i_{j,*}\mathbb Z/2\arrow{r}\arrow{d} & 0\\
0\arrow{r} &g_{!}\left((f_*\mathbb Z/2)|_{V} \right)\arrow{r} & f_*\mathbb Z/2\arrow{r} & \underset{1\leq j\leq n}\bigoplus \left(f_*\mathbb Z/2 \right)_{p_j}\arrow{r} & 0.
			\end{tikzcd}
            \end{equation}
By Theorem \ref{thm:coarse-map-closed}, the map $f\colon \vert \mathcal X(\mathbb R)\vert \rightarrow K$ is closed with finite fibers, and hence the formation of $f_\ast\ZZ/2$ commutes with arbitrary base change. Consequently, since the map $f^{-1}(V)\rightarrow V$ is an homeomorphism, the left vertical map in \eqref{align:diagram} is an isomorphism. Moreover, since $\{p_1,\dots, p_n\}$ is a finite number of points and $f^{-1}(p_j)\simeq H^1(G, \Aut(x_j))$ (see Proposition \ref{prop:fiber-H1}), we can rewrite diagram \eqref{align:diagram} as follows:
\begin{center}
	\begin{tikzcd}
0\arrow{r} &g_{!}\mathbb Z/2\arrow{r}\arrow[equal]{d} & \mathbb Z/2\arrow{r}\arrow{d} & \underset{1\leq j\leq n}{\bigoplus}i_{j,*}\mathbb Z/2\arrow{r}\arrow{d} & 0\\
0\arrow{r} &g_{!}\mathbb Z/2\arrow{r} & f_*\mathbb Z/2\arrow{r} & \underset{1\leq j\leq n}{\bigoplus}i_{j,*}\left((\mathbb Z/2)^{H^1(G, \Aut(x_j))} \right)\arrow{r} & 0,
			\end{tikzcd}
	\end{center} 
    where the vertical map on the right is induced by the diagonal inclusion $\mathbb Z/2\rightarrow (\mathbb Z/2)^{H^1(G, \Aut(x_j))}$.
    Passing to cohomology with coefficients in $\ZZ/2$, it yields a commutative diagram with exact rows
    \begin{center}
	\begin{tikzcd}[column sep=tiny]
0\arrow{r} &H^0_c(V)\arrow{r}\arrow[equal]{d} & H^0(K)\arrow{r}\arrow{d} & \underset{1\leq j\leq n}{\bigoplus}\mathbb Z/2\arrow{r}\arrow[hook]{d} & H^1_c(V)\arrow{r}\arrow[equal]{d}\arrow{r}\arrow[equal]{d}& H^1(K)\arrow{d}\arrow{r} &0\\
0\arrow{r} &H^0_c(V)\arrow{r} & H^0(\vert \mathcal X(\mathbb R)\vert )\arrow{r} & \underset{1\leq j\leq n}{\bigoplus} (\mathbb Z/2)^{H^1(G, \Aut(x_j))}\arrow{r} &  H^1_c(V)\arrow{r}& H^1(\vert \mathcal X(\mathbb R)\vert )\arrow{r}&  0,
			\end{tikzcd}
	\end{center} 
    and isomorphisms 
    \begin{equation}\label{eq : isogeq2}
         H^r(\vert \mathcal X(\mathbb R)\vert )\simeq H^r(K) \qquad \text{ for } r\geq 2.
         \end{equation}
    In particular, the map $H^1(K)\rightarrow H^1(\vert \mathcal X(\mathbb R)\vert )$ is surjective, 
and there is a commutative diagram with exact rows
\begin{center}
        	\begin{tikzcd}[column sep=tiny]
0\arrow{r} & H^0(K)/H^0_c(V)\arrow{r}\arrow{d} & \underset{1\leq j\leq n}{\bigoplus}\mathbb Z/2\arrow{r}\arrow[hook]{d} & \Ker(H^1_c(V)\rightarrow H^1(K))\arrow[hook]{d}\arrow{r} &0\\ 
0\arrow{r} & H^0(\vert \mathcal X(\mathbb R)\vert )/H^0_c(V)\arrow{r} & \underset{1\leq j\leq n}{\bigoplus} (\mathbb Z/2)^{H^1(G, \Aut(x_j))}\arrow{r}&  \Ker(H^1_c(V)\rightarrow H^1(\vert \mathcal X(\mathbb R)\vert ))\arrow{r}&  0,
			\end{tikzcd}
	\end{center} 
    in which the middle and the right vertical maps are injective. This shows that 
{\footnotesize
\begin{align}\label{eq : ineq0}
\begin{split}
    \dim(H^0&(\vert \mathcal X(\mathbb R)\vert))\\ &= \dim H^0_c(V) + \sum_{j = 1}^n \# H^1(G, \Aut(x_j))-\dim \Ker\left(H^1_c(V)\rightarrow H^1(\vert \mathcal X(\mathbb R)\vert ) \right) \\
    & \leq \dim H^0_c(V) + \sum_{j = 1}^n \# H^1(G, \Aut(x_j)) - \dim \Ker\left(H^1_c(V)\rightarrow H^1(K) \right)
    \\
    &= \dim(H^0(K))+\sum_{1\leq j\leq n}(\# H^1(G, \Aut(x_j))-1).
\end{split}
\end{align}
}
Combining the surjectivity of $H^1(K)\rightarrow H^1(\vert \mathcal X(\mathbb R)\vert )$ with 
\eqref{eq : isogeq2} and \eqref{eq : ineq0}, we get
\begin{equation}\label{eq : inK}
    h^*(\vert \mathcal X(\mathbb R)\vert )\leq h^*(K)+\sum_{1\leq j\leq n}(\# H^1(G, \Aut(x_j))-1). 
\end{equation}
We claim that $K$ is a union of connected components of $\vert \ca X(\mathbb C)\vert^G$. To prove this, let $x \in K \subset \va{\ca X(\CC)}^G$; it suffices to show that the connected component $C_x$ of $\va{\ca X(\CC)}^G$ containing $x$ is contained in $K$. The map $f^{-1}(U(\RR)) \to U(\RR)$ is a homeomorphism, so that 
$C_x$ contains a dense open subset which lies in the image of the map $f^{-1}(C_x) \to C_x$; as this map is closed by Theorem \ref{thm:coarse-map-closed}, it must be surjective, proving the claim.

Consequently, $h^\ast(K) \leq h^\ast(\va{\ca X(\CC)}^G)$. In turn, by the classical Smith--Thom inequality  \eqref{align:inequality-ST}, one has $h^\ast(\va{\ca X(\CC)}^G) \leq h^\ast(\va{\ca X(\CC)})$. In view of \eqref{eq : inK}, this yields
\begin{align*}
     h^*(\vert \mathcal X(\mathbb R)\vert )& \leq h^*(\vert \ca X(\mathbb C)\vert^G)+\sum_{1\leq j\leq n}(\# H^1(G, \Aut(x_j))-1) \\
     &\leq h^*(\vert \ca X(\mathbb C)\vert)+\sum_{1\leq j\leq m}(\# H^1(G, \Aut(x_j))-1)\\
& \leq h^*(\vert \ca X(\mathbb C)\vert)+\sum_{1\leq j\leq m}(\#(\Aut(x_j)/\Aut(x_j))-1)\\
&= h^*(\vert \mathcal I_{\mathcal X}(\mathbb C)\vert),
\end{align*}
where the second to last inequality follows from \cite[Lemma 6.1]{AdGF-grup}, and the last equality from 
Proposition \ref{prop:inertia-description} and Lemma \ref{lemma:local-structure}. This concludes the proof.
   \end{proof}

\section{Smith--Thom for real stacky curves}\label{sec:STforcurves}
The goal of this section is to prove Theorem \ref{thm:introST-curves}. 
We compute \( h^*( \va{\ca I_{[X/\Gamma]}(\mathbb{C}))} \) in Section \ref{sec:inertiacurve}, and \( h^*(\lvert [X/\Gamma](\mathbb{R}) \rvert) \) in Section \ref{sec:realpointscurve}. We then combine these computations in Section \ref{sec:proofsmiththomcurve} to establish Theorem \ref{thm:introST-curves}.

\subsection{Inertia of complex stacky curves.}\label{sec:inertiacurve}
Let $X$ be a smooth curve over $\CC$, and let $\Gamma$ be a finite group which acts on $X$ over $\CC$. We let $K \subset \Gamma$ be the kernel of the homomorphism $\Gamma \to \Aut_\CC(X)$ associated to the $\Gamma$-action, and define $Q \coloneqq \Gamma/K$. 
The restriction of the action on $X$ of $\Gamma$ to $K$ yields the trivial action of $K$ on $X$, and the induced action of $Q$ on $X$ is faithful. Let $$M \coloneqq X/Q=X/\Gamma,$$
which is the coarse moduli space of both $[X/Q]$ and $[X/\Gamma]$. Let $I_{[X/\Gamma]}$ (resp.\ $I_{[X/Q]}$) be the coarse moduli space of the inertia stack $\mathcal I_{[X/\Gamma]}$ (resp.\ $\mathcal I_{[X/Q]}$) of $[X/\Gamma]$ (resp.\ $[X/Q]$).
\begin{proposition} \label{prop:inertia-curves-description}
In the above notation, assume that the subgroup $K \subset \Gamma$ is contained in the center of $\Gamma$, so that for every $x\in X(\mathbb C)$ there is a natural inclusion $K\subset \Gamma_{x}/\Gamma_{x}$. Let $\Delta \subset M(\CC)$ be the branch locus of the quotient map $q \colon X(\CC) \to M$, and choose a lift $y_x \in X(\CC)$ of each $x \in \Delta$. Then there is a canonical homeomorphism
    \[
I_{[X/\Gamma]}(\CC)=   \vert \mathcal I_{[X/\Gamma]}(\CC) \vert \simeq \left(K \times M(\CC) \right) \coprod \coprod_{x \in \Delta} \left( \Gamma_{y_x}/\Gamma_{y_x} - K \right) 
    \]
    that commutes with the natural projections onto $M(\mathbb C)$. 
\end{proposition}
\begin{proof}
We may assume that $X$ is connected. It suffices to show that the map $I_{[X/\Gamma]} \to M$ has $\# K$ disjoint sections. Indeed, $I_{[X/\Gamma]} \to M$ is finite by \cite[Lemma 5.2]{AdGF-grup}, hence for each irreducible component $Z \subset I_{[X/\Gamma]}$ of dimension one, the restriction $Z \to M$ is a finite morphism of irreducible curves, hence an isomorphism if it admits a section; moreover, over the open subset of $M$ where the stabilizer group is exactly $K$, the fibers of $I_{[X/\Gamma]} \to M$ have cardinality exactly $\# K$ by Proposition \ref{prop:inertia-description}.

Let $S \subset X \times_\CC \Gamma$ be the stabilizer group scheme associated to the action of $\Gamma$ on $X$ over $\CC$, so that $S$ can be described pointwise as
$S = \set{(x, g) \in X \times_\CC \Gamma \mid g \cdot x = x}$. Then $\Gamma$ acts on $S$ by $\gamma \cdot (x,g) = (\gamma \cdot x, \gamma g \gamma^{-1})$ for $\gamma \in \Gamma$ and $(x,g) \in S$. Moreover, we have a canonical isomorphism
$\ca I_{[X/\Gamma]} = [S/\Gamma]$ (see e.g. \cite[Exercise 3.2.12]{notesstack}), so that, in particular, $I_{[X/\Gamma]} = S/\Gamma$.

Since $K$ is contained in the center of $\Gamma$, to any $k \in K$ one can associate the following well defined section $s_k$ of the canonical map $ I_{[X/\Gamma]} \to M$: 
\begin{align}\label{align:sec-k}
s_k \colon  M \to  I_{[X/\Gamma]} = S/\Gamma, \quad \quad [x] \mapsto [(x, k)].
\end{align}
The sections $s_k$ and $s_{k'}$ are disjoint for $k \neq k' \in K$, and so the proposition follows. 
\end{proof}

\begin{proposition} \label{prop:betti-nr-inertia-abelian}
Assume that $K \subset \Gamma$ is contained in the center of $\Gamma$.  
Then
    \begin{align} \label{eq:FXGam:1}
    \begin{split}
    h^\ast(I_{[X/\Gamma]}(\mathbb C)) = \# K \cdot h^\ast(M(\CC))   + \left(\sum_{x \in \Delta}\# (\Gamma_{y_x}/\Gamma_{y_x})\right) - \# \Delta\cdot \# K.
    \end{split}
    \end{align}
    If, in addition, $\Gamma_{y_x}$ is abelian for each $x \in \Delta$, then
    \begin{align}
    \label{eq:FXGam:2}
     h^\ast( I_{[X/\Gamma]}(\mathbb C)) = \# K  \cdot h^\ast(I_{[X/Q]}(\mathbb C)). 
    \end{align}
\end{proposition}
\begin{proof}
By Proposition \ref{prop:inertia-curves-description}, we have
 \begin{align} \label{align:dim-coh-inertia}
    h^\ast(I_{[X/\Gamma]}(\mathbb C)) &= \# K \cdot  h^\ast(M(\CC)) + \sum_{x \in \Delta}\left(\# (\Gamma_{y_x}/\Gamma_{y_x})  - \# K\right), 
\end{align}
and \eqref{align:dim-coh-inertia} implies \eqref{eq:FXGam:1}. Applying \eqref{align:dim-coh-inertia} to the quotient stack $[X/Q]$ gives
\begin{align}\label{align:QKGamma:2}
    h^\ast(I_{[X/Q]}(\mathbb C))  &=  h^\ast(M(\CC)) + \sum_{x \in \Delta}\left(\# (Q_{y_x}/Q_{y_x}) - 1 \right). 
\end{align}
If $\Gamma_{y_x}$ is abelian for each $x \in \Delta$, then one has
$\Gamma_{y_x}/\Gamma_{y_x}=\Gamma_{y_x}$, $Q_{y_x}/Q_{y_x}=Q_{y_x}$ and $\#K \cdot \# Q_{y_x} = \# \Gamma_{y_x}.$
Therefore, \eqref{eq:FXGam:2} follows from \eqref{align:dim-coh-inertia} and \eqref{align:QKGamma:2}, and we are done. 
\end{proof}

\begin{example}
Consider the moduli stack $\ca A_1$ of elliptic curves over $\CC$, with coarse moduli space $\ca A_1 \to \mathsf A_1 = \bb A^1_\CC$. 
Then $h^\ast(\vert \mathcal I_{\mathcal A_1}(\CC)\vert ) = 8$. Indeed, we let $\ell \geq 3$ be a prime number and let $\ca A_1[\ell]$ be the moduli space of elliptic curves with level $\ell$ structure; it is equipped with a $\SL_2(\bb F_\ell)$-action such that $\ca A_1 = [\SL_2(\bb F_\ell) \sm \ca A_1[\ell]]$. In this case, we have $K = \langle -1 \cdot \Id \rangle \subset \SL_2(\bb F_\ell) = \Gamma$, and $\Gamma/K = \PSL_2(\bb F_\ell)$. The locus $\Delta \subset \mathsf A_1(\CC)$ of isomorphism classes of elliptic curves with automorphism group larger than $\set{\pm 1}$ consists of two points, with respective automorphism groups $\ZZ/4$ and $\ZZ/6$. Thus, Proposition \ref{prop:betti-nr-inertia-abelian} implies that $h^\ast(\vert \mathcal I_{\mathcal A_1}(\CC)\vert ) = 2 + 4 + 6 - 2 \cdot 2 = 8$. 
\end{example}
\subsection{Topology of real stacky curves.}\label{sec:realpointscurve}
In this section, we study the topology of $\va{\ca X(\RR)}$ when $\ca X$ is a stacky quotient curve over $\RR$. 
\subsubsection{Local geometry.} 
Let $\ca X$ be a separated Deligne--Mumford stack of finite type over $\RR$. Let $f \colon \va{\ca X(\RR)} \to M(\RR)$ be the map induced by the coarse moduli space $\ca X \to M$. 
\begin{definition}\label{def:h1G}For $x \in M(\RR)$, define $h^1(G,x) = \# \rm H^1(G, \Aut(z_\CC))$ where $z \in \ca X(\RR)$ is an object whose isomorphism class $[z] \in \va{\ca X(\RR)}$ lies in the inverse image $f^{-1}(x)$ of $x$. By Proposition \ref{prop:fiber-H1}, we have $h^1(G,x) = \# f^{-1}(x)$. 
\end{definition}

For example, when $\ca X$ is the quotient $\ca X = [X/H]$ of a real variety $X$ by finite group scheme $H$ over $\RR$, and when $x \in (X/H)(\RR)$ is the image of a point $y \in X(\RR)$ under the map $X(\RR) \to (X/H)(\RR)$, then $h^1(G,x) = \#\rm H^1(G, H_y(\CC))$. This definition allows us to study the local topology of a stacky curve around a point with non-trivial stabilizer.
    \begin{lemma} \label{lem:local-groupstructure-curves}
   Let $X$ be a smooth curve over $\RR$. Let $H$ be a finite $\RR$-group scheme that acts faithfully on $X$ over $\RR$. 
   \begin{enumerate}\item For each $x \in X(\RR)$, there exists an integer $n \geq 1$ such that the stabilizer group scheme $H_x$ is isomorphic to $\mu_n$. 
        \item For $x \in (X/H)(\RR)$, the number $ h^1(G,x)$ is equal to $1$ (resp.\ $2$) if $n$ is odd (resp.\ even).
        \end{enumerate}
    \end{lemma}
\begin{proof}
Let $\Gamma = H(\CC)$. Since the action of $\Gamma$ is faithful, there are only finitely many points $x \in X(\RR)$ 
 with non-trivial stabilizer $\Gamma_x$. Since the statement is trivial for points with trivial stabilizer, we focus on the points $x$ with  $\Gamma_x\neq 0$. Choose an open neighbourhood $U$ of $x$, stable under the actions of $G$ and $\Gamma$, not containing any other point with non-trivial stabilizer, and $G$-equivariantly biholomorphic to an open disk centered in $x$ endowed with the $G$-action given by complex conjugation. Since the group of biholomorphisms of the disk with one fixed point is isomorphic to $\mathbb S^1$, we see that $\Gamma_x$ is isomorphic to $\ZZ/n$ for some positive integer $n$. Moreover, a generator $\gamma \in \Gamma_x$ acts as $\gamma(z)=e^{i\theta}z$ if $z$ is a local coordinate around $x$. Since the $G$-action is compatible with the action of $\Gamma$, this forces a $G$-equivariant isomorphism $\Gamma_x\simeq \mu_n$, proving the first item of the lemma. By the fact that $\# \rm H^1(G,\mu_n)$ equals $1$ if $n$ is odd and $2$ if $n$ is even, the second item follows from the first. 
\end{proof}
\subsubsection{Global geometry.} \label{subsec:fin-quot-real-curves}
We now study the possible shapes of the connected components of the real locus of a real stacky quotient curve. 
\begin{proposition} \label{prop:top-reallocus-stcurve}
    Let $X$ be a smooth curve over $\RR$. Let $H$ be a finite \'etale group scheme over $\RR$ which acts on $X$ over $\RR$. Let $C$ be a connected component of $\va{[X/H](\RR)}$. Then $C$ is homeomorphic to either an interval in $\RR$ of the form $(0,1), (0,1]$ or $[0,1]$, or to the circle $\mathbb S^1$. If $X$ is proper then only the possilibities $[0,1]$ and $\mathbb S^1$ can occur. 
\end{proposition}
We actually prove something slighlty more general in the following Lemma \ref{lem:isometries-circle}. Observe that  Proposition \ref{prop:top-reallocus-stcurve} follows from Theorem \ref{thm:introlemme-prefere} and Lemma \ref{lem:isometries-circle}.

\begin{lemma}\label{lem:isometries-circle}
 Let $X$ be a smooth curve over $\RR$. Let $H$ be a finite \'etale group scheme over $\RR$ acting on $X$ over $\RR$. Each connected component of $X(\RR)/H(\RR)$ is homeomorphic to the interval $[0,1] \subset \RR$, the interval $(0,1]$, the interval $(0,1)$, or the circle $\mathbb S^1 \subset \RR^2$. 
\end{lemma}
\begin{proof}
We may and do assume that $H$ acts faithfully on $X$. 

First, assume that $X$ is proper, so that $X(\mathbb R)$ is compact, and let $C$ be a connected component of $X(\RR)$ with stabilizer $\Stab_{H(\RR)}(C)$ in $H(\RR)$. 
We start by proving that 
\begin{equation}\label{eq:connectedcomponent}
C/\Stab_{H(\RR)}(C)\simeq \mathbb S^1 \quad\text{or} \quad C/\Stab_{H(\RR)}(C)\simeq [0,1].
\end{equation}
Recall that every connected Riemann surface $S$ admits a unique complete Riemann metric $g$ with constant curvature being negative (genus $\geq 2$), zero (genus zero) or positive (genus one). 
Moreover, for genus $\geq 2$ the group $\rm{Bihol}(S)$ coincides with the group $\Isom(S,g)^+$ of orientation preserving isometries of the Riemannian manifold $(S,g)$. In genus zero we have, for the subgroup $\PGL_2(\RR) \subset \PGL_2(\CC) = \rm{Bihol}(\PP^1(\CC))$, that $\PGL_2(\RR) = \rm{SO}_3(\RR)$ acts as isometries on $\PP^1(\CC) \simeq S^2$. The automorphism group of any complex elliptic curve preserves its Riemannian metric. 

In particular, as $H$ acts faithfully on $X$, there are natural inclusions
$$\Stab_{H(\RR)}(C)\subset H(\mathbb R)\subset \Isom(X(\CC)) \subset \Homeo(X(\CC)),$$ 
where $\Isom(X(\CC))$ is the group of isometries with respect to the Riemannian metric of $X(\CC)$. 
We endow $C$ with the Riemannian metric induced by the embedding $C \subset X(\CC)$. Then $C$ is a compact one-dimensional Riemannian manifold, and hence isometric to a circle of some length $L$: we have $C \simeq \RR/L\ZZ$ with the standard Riemannian metric. In particular, $\rm{Isom}(C) \simeq \rm{O}(2)$. By the above, $\Stab_{H(\RR)}(C) \subset \rm{Isom}(X(\CC))$, and hence $\Stab_{H(\RR)}(C)  \subset \rm{Isom}(C) \simeq \rm{O}(2)$. So, $\Stab_{H(\RR)}(C)$ is a finite subgroup of $\rm{O}(2) = \rm{Isom}(\mathbb S^1)$ with $\mathbb S^1 = \set{z \in \CC \mid \va{z} = 1}$, hence it is generated by multiplication by some root of unity and, possibly, the standard complex conjugation on $\mathbb S^1$. This yields \eqref{eq:connectedcomponent}.

Let $C_1, \dotsc, C_n$ be the connected components of $X(\RR)$; each $C_i$ is homeomorphic to $\mathbb S^1$. Then $I \coloneqq \set{1, \dotsc, n}$ admits a partition $I = I_1 \sqcup I_2 \sqcup \cdots \sqcup I_k$ with $k \leq n$ such that the $I_j$ are the orbits for the induced action of $H(\RR)$ on $I$. For each $j \in \set{1, \dotsc, k}$, choose an element $i_j \in I_j$. 
Let $H(\RR)_j = \rm{\Stab}_{H(\RR)}(C_{i_j})$ be the stabilizer of $C_{i_j}$ in the group $H(\RR)$. 
Then
\begin{align*}
X(\RR)/H(\RR) = \left(\coprod_{i = 1}^n C_i \right)/H(\RR) = \coprod_{j = 1}^k\left(\left(\coprod_{i \in I_j} C_i \right)/H(\RR) \right)  \simeq \coprod_{j = 1}^k C_{i_j}/H(\RR)_j. 
\end{align*}
Thus, the lemma in the case where $X$ is proper follows from \eqref{eq:connectedcomponent}.

In the general case, consider the smooth projective compactification $X \hookrightarrow Y$ of $X$. The action of $H$ on $X$ extends to an action of $H$ on $Y$, and the natural map $X(\RR)/H(\RR) \to Y(\RR)/H(\RR)$ is an open embedding whose complement is a finite set (possibly empty). By what has already been proved, each connected component of $Y(\RR)/H(\RR)$ is homeomorphic to $[0,1]$ or $\mathbb S^1$. By removing the points in $\Delta(\RR) \subset Y(\RR)$, where $\Delta = Y - X$, we see that each connected component of $X(\RR)/H(\RR)$ is homeomorphic to $[0,1]$, $(0,1]$, $(0,1)$ or $\mathbb S^1$, and we are done. 
\end{proof}
\subsubsection{Map to the coarse moduli space.}\label{sec:maptocoarsmodulispace}
Finally, we study the map from a real stacky curve to its coarse moduli space. 

   Let $X$ be a smooth curve over $\RR$. Let $H$ be a finite $\RR$-group scheme that acts on $X$ over $\RR$, with associated real structure $\sigma \colon H(\CC) \to H(\CC)$. 
   Let $p \colon [X/H] \to X/H = M$ be the coarse moduli space map, with induced map $f \colon \va{[X/H](\RR)} \to (X/H)(\RR) = M(\RR)$.
    \begin{lemma} \label{lem:topXR-faithful}
    Let $C \subset M(\RR)$ be a connected component. The following holds.
    \begin{enumerate}
        \item \label{IT:IT:0} The map $f \colon \va{[X/H](\RR)} \to M(\RR)$ is closed.
        \item \label{IT:IT:1}
        If $H$ acts faithfully on $X$, then the map $f^{-1}(C)\to C$ is surjective. 
        \item \label{IT:IT:2}Assume that $H$ acts faithfully on $X$, and that for each $m \in C$, we have $h^1(G,m) = 1$, with $h^1(G,m)$ as in Definition \ref{def:h1G}. Then $f^{-1}(C) \to C$ is a homeomorphism. 
    \end{enumerate}
        \end{lemma}
        \begin{proof} 
    Item \ref{IT:IT:0} follows from Theorem \ref{thm:coarse-map-closed}. 
          To prove item \ref{IT:IT:1}, note that since the action is faithful and $X$ is smooth, $f \colon \va{[X/H](\RR)} \to M(\RR)$ is surjective and closed; indeed, $f$ is closed by item \ref{IT:IT:0} and its image contains a dense open subset, so $f$ is surjective. In particular, item \ref{IT:IT:1} follows. Moreover, for $m \in C$, we have $\# f^{-1}(m) = h^1(G,m)$, see Proposition \ref{prop:fiber-H1}. Hence item \ref{IT:IT:2} follows from items \ref{IT:IT:0} and \ref{IT:IT:1}.  
        \end{proof}
\begin{proposition} \label{prop:topXR-faithful}
    Assume that $H$ acts faithfully on $X$ over $\RR$. Let $C \subset M(\RR)$ be a connected component and let 
 $\mr S = \set{x_1, \dotsc, x_n}\subset C$ be the finite set of points such that $h^1(G,x_i)\neq 1$. Assume that $\mr S\neq \emptyset$. 
        \begin{enumerate}
         \item Assume that  $C$ is an open interval, and fix a homeomorphism $\varphi\colon  C\xrightarrow{\sim} (0,1)$. There exists an homeomorphism $$
    \psi \colon f^{-1}(C) \xlongrightarrow{\sim} (0,y_1] \coprod [y_1, y_2] \coprod \cdots \coprod [y_{n-1},y_n] \coprod [y_n,1)$$
        such that the following diagram commutes:
        \[
        \xymatrixcolsep{5pc}
        \xymatrix{
        f^{-1}(C) \ar[d] \ar[r]^-\psi & (0,y_1] \sqcup [y_1, y_2] \coprod \cdots \coprod [y_{n-1},y_n] \coprod [y_n,1) \ar[d] \\
        C \ar[r]^-{\varphi} & (0,1),
        }
        \]
        where the vertical arrows are the canonical ones and $y_i=\varphi(x_i)$. 
\item Assume that $C$ is a circle, and fix a homeomorphism $\varphi \colon C \xrightarrow{\sim} \mathbb S^1$ with $\varphi(x_1)=1$. Then there exists a homeomorphism $$\psi \colon f^{-1}(C) \xlongrightarrow{\sim} [0, \theta_1] \coprod [\theta_{1},\theta_2]\coprod \cdots \coprod [\theta_{n-1}, 1]$$
  such that the following diagram commutes:
        \[
        \xymatrixcolsep{5pc}
        \xymatrix{
        f^{-1}(C) \ar[d] \ar[r]^-\psi & [0, \theta_1] \coprod [\theta_{1},\theta_2]\coprod \cdots \coprod [\theta_{n-1}, 1] \ar[d]^-{x\mapsto e^{2\pi ix}} \\
        C \ar[r]^-{\varphi} & \mathbb S^1,
        }
        \]
        where the vertical arrows are the canonical ones and $\varphi(x_j) = e^{i \theta_j}$ for $j\neq 1$. 
    \end{enumerate}
\end{proposition}
\begin{proof}
By Lemma \ref{lem:topXR-faithful}, the map $f \colon f^{-1}(C) \to C$ is surjective. 
By Proposition \ref{prop:fiber-H1} and Lemma \ref{lem:local-groupstructure-curves}, for each connected component $K \subset M(\RR)$, the map $f^{-1}(K) \to K$ is an isomorphism outside $\mr S \subset K$ and has two fibers above each point of $\mr S$. The proposition follows from this and from Proposition \ref{prop:top-reallocus-stcurve}. 
\end{proof}
The content of Proposition \ref{prop:topXR-faithful} is illustrated in Figure \ref{fig: f-1C->C} below. The picture on the left shows the case in which $C$ is a circle: $C$ is depicted as the internal circle and $f^{-1}(C)$ as the union of the three external arcs $[0,\theta_1]$, $[\theta_1,\theta_2]$ and $[\theta_2,1]$. The picture on the right shows the case in which $C$ is a open interval: $C$ is depicted as the bottom interval and $f^{-1}(C)$ as the disjoint union of the three upper segments $(0,y_1]$, $[y_1,y_2]$ and $[y_2,1)$.

 \begin{center}
\begin{figure}[h!]
\begin{center}
\begin{picture}(0,130)
\put(36,25){$0$}
\put(82,25){$y_1$}
\put(153,25){$y_2$}
\put(210,25){$1$}

\put(-126,30){$1$}
\put(-105,75){$e^{i \theta_2}$}
\put(-155,75){$e^{i \theta_1}$}
\put(-140,0){$C\simeq \mathbb S^1$}
\put(100,0){$C\simeq (0,1)$}
\put(-170,15){\includegraphics[width=0.2\textwidth]{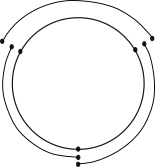}}
\put(30,15){\includegraphics[width=0.4\textwidth]{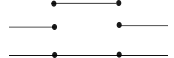}}
\end{picture}
\caption{The morphism $f^{-1}(C)\rightarrow C$.}
\label{fig: f-1C->C}
\end{center}
\end{figure}
\end{center}
\subsection{Smith--Thom for real stacky curves.}\label{sec:proofsmiththomcurve} We are now ready to proceed with the:
\begin{proof}[Proof of Theorem \ref{thm:introST-curves}]
The action of $H$ on $X$ corresponds to a homomorphism
\begin{align}\label{align:action-hom}
H \to \underline{\Aut}_\RR(X),
\end{align}
where the latter denotes the automorphism group scheme of $X$ over $\RR$. Let $K \subset H$ be the kernel of \eqref{align:action-hom}, and let $Q = H/K$ be the quotient of $H$ by $K$. The canonical map $Q \to \underline{\Aut}_\RR(X)$ is a closed immersion, and $Q(\CC)$ acts faithfully on $X(\CC)$. 
Let $[X/H] \to M$ be the coarse moduli space; we have $M(\CC) = X(\CC)/H(\CC) = X(\CC)/Q(\CC)$. 
\\
\\
\textbf{Step 1:} \emph{If $H$ is abelian, and if the Smith--Thom inequality \eqref{align:inequality-ST-stacks} holds for the quotient stack $[X/Q]$, then it also holds for $[X/H]$.}
\begin{proof}[Proof of Step 1]
Assume the 
Smith--Thom inequality \eqref{align:inequality-ST-stacks} for $[X/Q]$. 
Consider the canonical map 
\begin{align}\label{align:themapg}
g \colon \va{[X/H](\RR)} \to \va{[X/Q](\RR)}.
\end{align}
We claim that \begin{align} \label{align:we have}\# g^{-1}(p)\leq \#K(\mathbb C) \quad \quad \text{ for any 
$p \in \va{[X/Q](\RR)}$.
}
\end{align} To prove this, choose a complete set of representatives $\rm R_H \subset \rm Z^1(G,H)$ for $\rm H^1(G,H)$. Let $\rm I_Q = m(\rm{R}_H) \subset \rm Z^1(G,Q)$ be the image of $\rm R_H$ in $\rm Z^1(G,Q)$ under the natural map $m \colon \rm Z^1(G,H) \to \rm Z^1(G,Q)$, and extend $\rm I_Q$ to a complete set of representatives $\rm I_Q \subset \rm R_Q \subset \rm Z^1(G,Q)$  for $\rm H^1(G,Q)$. By construction, the map $m \colon \rm Z^1(G,H) \to \rm Z^1(G,Q)$ restricts to a map $m\colon \rm R_H \to \rm  R_Q$. By Theorem \ref{thm:introlemme-prefere}, of which we retain the notation,  we have a commutative diagram 
\[
\xymatrix{
\va{[X/H](\RR)} \ar[d]^{\wr}  \ar[r]^-g & \va{[X/Q](\RR)} \ar[d]^{\wr}\\
\coprod_{\gamma \in \rm  R_H} X_{\gamma}(\RR)/H_{\gamma}(\RR) \ar[r] & \coprod_{\mu \in \rm  R_Q} X_{\mu}(\RR)/Q_{\mu}(\RR)
}
\]
in which the vertical arrows are homeomorphisms and where the map on the bottom is induced by the map $m\colon\rm R_H \to \rm  R_Q$ and the quotient morphisms
$X_{\gamma}(\RR)/H_{\gamma}(\RR)\rightarrow X_{m(\gamma)}(\RR)/Q_{m(\gamma)}(\RR).$
To prove the claim, we may assume that $p = g(q)$ for some $q \in \va{[X/H](\RR)}$. Let $\gamma \in \rm R_H$ such that $q \in X_\gamma(\RR)/H_\gamma(\RR)$. 
The exact sequence of abelian groups
\[
0 \to K_{\gamma}(\RR) \to H_{\gamma}(\RR) \to Q_{m(\gamma)}(\RR) \to \rm  H^1(G,K_{\gamma}) \to \rm  H^1(G,H_{\gamma}) \to \rm  H^1(G,Q_{m(\gamma)}),
\]
where $K_{\gamma}\coloneqq \Ker(H_{\gamma}\rightarrow Q_{m(\gamma)})$, 
shows that $\#g^{-1}(p)\leq \#\rm  H^1(G, K_{\gamma}(\mathbb C))\leq \#K(\mathbb C)$. This proves \eqref{align:we have}.

By Proposition \ref{prop:top-reallocus-stcurve}, each connected component of $\va{[X/H](\RR)}$ or $\va{[X/Q](\RR)}$ is homeomorphic to a circle or an interval. Let $C$ be a connected component of $\va{[X/Q](\RR)}$. If $C$ is a circle, then, by \eqref{align:we have}, the inverse image $g^{-1}(C)$ consists of at most $\# K$ connected components which are circles or intervals. This implies that 
\begin{align} \label{align:CIRCLE}
h^\ast(g^{-1}(C)) \leq 2 \cdot \# K(\mathbb C) \quad \quad \text{ if $C$ is a circle.}
\end{align}
Note that every surjective continuous map from a circle to an interval has the property that the fiber above any point in the interior of the interval has cardinality at least two. To prove this, take a point $p$ not on boundary of the interval. If $p$ had only one pre-image, then, eliminating the point, we would get a surjective morphism from a connected space, a circle minus one point, to a non-connected one, an interval minus one point not on the boundary. This is not possible, so that $p$ has at least two pre-images.

Thus, by \eqref{align:we have}, if the connected component $C$ of $\va{[X/Q](\RR)}$ is an interval, then $g^{-1}(C)$ is an union of $a$ intervals and $b$ circles with $a+2b\leq \#K$. Hence, we have: 
\begin{align} \label{align:INTERVAL}
h^\ast(g^{-1}(C)) \leq \# K(\mathbb C) \quad \quad \text{ if $C$ is an interval.}
\end{align}
Therefore, we have: 
\begin{align*}
h^\ast(\va{[X/H](\RR)}) &= \sum_{C \in \pi_0(\va{[X/Q](\RR)})}h^\ast\left( g^{-1}(C) \right)\\
& = 
\sum_{C \text{ circle}}h^\ast\left( g^{-1}(C) \right)
+
\sum_{C \text{ interval}}h^\ast\left( g^{-1}(C) \right) \\
& 
\stackrel{\text{(a)}}{\leq}  
\sum_{C \text{ circle}} 2 \cdot \# K(\CC)
+
\sum_{C \text{ interval}} \# \rm   K(\CC) \\
& = 
\# K(\CC) \cdot h^\ast(\va{[X/Q](\RR)}) \\
& \stackrel{\text{(b)}}{\leq}  \# K(\CC) \cdot h^\ast(\vert \mathcal I_{[X/Q]}(\mathbb C)\vert) \\
& \stackrel{\text{(c)}}{=} h^\ast(\vert \mathcal I_{[X/H]}(\mathbb C)\vert),
\end{align*}
where (a) holds by \eqref{align:CIRCLE} and \eqref{align:INTERVAL}, (b) by the fact that the Smith--Thom inequality \eqref{align:inequality-ST-stacks} holds for $[X/Q]$ by assumption, while $(c)$ holds by Proposition \ref{prop:betti-nr-inertia-abelian}, which we can apply since $H$ is abelian. This proves Step 1. 
\end{proof}
Let $\Delta \subset X(\CC)/Q$ be the branch locus of the quotient map $q \colon X(\CC) \to X(\CC)/Q$. For each $x \in \Delta$ choose an element $y_x \in X(\CC)$ such that $q(y_x) = x$. Define
\[
\Delta' \coloneqq \set{x \in \Delta \cap (X/Q)(\RR) \mid h^1(G,x) > 1},
\]
where $h^1(G,x) = \# \rm  H^1(G, H_y(\CC))$ for some $y \in q^{-1}(x)$.\\
\\
\textbf{Step 2:} \emph{The Smith--Thom inequality \eqref{align:inequality-ST-stacks} holds when $H = Q$, that is, when the action of $H$ on $X$ over $\RR$ is faithful.}
\begin{proof}
Assume that $H = Q$ acts faithfully on $X$. By Lemma \ref{lem:topXR-faithful}, the natural map $f \colon \va{\ca X(\RR)} \to M(\RR)$ is surjective. Let $C \subset (X/Q)(\RR)$ be a connected component which is homeomorphic to a circle.   Since the action of $H$ is faithful, by Proposition \ref{prop:topXR-faithful}, $f^{-1}(C)$ is homeomorphic to a circle if $h^1(G,x) = 1$ for each $x \in C$, and $f^{-1}(C)$ is homeomorphic to the union of $\#(C \cap \Delta')$ intervals if $\Delta' \cap C \neq \emptyset$.   
    In particular, we have:
    \[
    h^\ast\left( f^{-1}(C) \right) =\begin{cases}
2 & \text{ if } \quad C \cap \Delta'= \emptyset, 
\\
    \# ( C \cap \Delta') & \text{ if } \quad  C \cap \Delta' \neq \emptyset.  
\end{cases}
    \]
 Let $I \subset (X/Q)(\RR)$ be a connected component which is homeomorphic to the open interval $(0,1)$. 
    Since the action of $H$ is faithful, by Proposition \ref{prop:topXR-faithful}, $f^{-1}(I)$ is homeomorphic to the union of  $\# ( I \cap \Delta') + 1$ intervals. In particular, we have $h^\ast( f^{-1}(I))=\# ( I \cap \Delta') + 1$. Therefore, we have:
\begin{align*}
h^\ast(\va{\ca X(\RR)}) &= \sum_{C \in \pi_0(M(\RR)) \text{ circle}} h^\ast(f^{-1}(C)) + \sum_{I \in \pi_0(M(\RR)) \text{ interval}} h^\ast(f^{-1}(C))  \\
& \stackrel{\text{(a)}}{=}  
\sum_{C \cap \Delta' = \emptyset} 2 + \sum_{C \cap \Delta' \neq \emptyset} \#(C \cap \Delta') + \sum_{I} \left( \#(C \cap \Delta') + 1 \right) \\
& = 
\left(\sum_{C \cap \Delta' = \emptyset} 2 + \sum_{I} 1\right) +
\left(\sum_{C \cap \Delta' \neq \emptyset} \#(C \cap \Delta')  +
\sum_{I \cap \Delta' \neq \emptyset} \#(C \cap \Delta') \right)
\\
& \leq h^\ast((X/Q)(\RR)) + \sum_{x \in \Delta} 1
 \\
 & \stackrel{\text{(b)}}{\leq}  h^\ast((X/Q)(\RR)) + \sum_{x \in \Delta} \left(\# (H_{y_x}(\CC)/H_{y_x}(\CC))- 1\right) \\
 & \stackrel{\text{(c)}}{\leq} h^\ast((X/Q)(\CC)) + \sum_{x \in \Delta} \left(\# (H_{y_x}(\CC)/H_{y_x}(\CC)) - 1\right)\\
 & \stackrel{\text{(d)}}{=} h^\ast(\vert \mathcal I_{\mathcal X}(\CC)\vert).
\end{align*}
Here, (a) follows from the previous discussion, (b) from the fact that we have $2 \leq \# (H_{y_x}(\CC)/H_{y_x}(\CC))$, (c) from the classical Smith--Thom inequality (\ref{align:inequality-ST}) for $X/Q$, and, finally, (d) from Proposition \ref{prop:betti-nr-inertia-abelian}. This proves Step 2. 
\end{proof}
By combining Steps 1 and 2, Theorem \ref{thm:introST-curves} follows.  
\end{proof}




\section{Topology of split gerbes over a real variety}\label{sec:classifyingstack}

Let \( U \) be a geometrically connected real variety with \( U(\mathbb{R}) \neq \emptyset \), and let \( H \to U \) be a finite étale group scheme, equipped with the trivial action on \( U \) over $U$. In this section, we describe how to compute the topology of \( \lvert [U/H](\mathbb{R}) \rvert \) by comparing it with \( U(\mathbb{R}) \). This analysis leads to the proof of Theorem \ref{thm:topcovering[U/H]} as well as Corollaries \ref{cor:gerbesovercurves} and \ref{cor:gerbesovercurvesconcrete}.

In Section \ref{sec:notation U/H}, we introduce notation that will be used throughout Section \ref{sec:classifyingstack}. The main work toward proving Theorem \ref{thm:topcovering[U/H]} is carried out in Section \ref{sec:topalgU/H}. In Section \ref{sec:proofU/H}, we assemble all the ingredients and prove Theorem \ref{thm:topcovering[U/H]} along with Corollaries \ref{cor:gerbesovercurves} and \ref{cor:gerbesovercurvesconcrete}.

\subsection{Notation and description of the complex inertia.}\label{sec:notation U/H}
Let \( U \) be a geometrically connected real variety with \( U(\mathbb{R}) \neq \emptyset \), and let \( H \to U \) be a finite étale group scheme, equipped with the trivial action on \( U \) over $U$. 
For $x\in U(\mathbb R)$, we let $H_x$ denote the fiber of $H \to U$ over $x$; thus $H_x$ is a finite \'etale group scheme over $\RR$. We let $\overline x \colon \Spec(\CC) \to U$ be the geometric point associated to $x$, and define $H_{\overline x} = H \times_{U}\overline x$. Thus, $H_{\overline x}$ is the constant group scheme over $\CC$ associated to the finite group $H_{x}(\mathbb C)$ which, by abuse of notation, we will also denote by $H_{\overline x}$. The finite group $H_{\overline x}$ is endowed with an action of $G$ whose associated involution we denote by  
\begin{align}\label{align:involution-G-Hx}
    \sigma_x \colon H_{\overline x}\rightarrow H_{\overline x}.
\end{align}
The natural map $U\rightarrow \mathcal X$ is a section of the coarse moduli space map $[U/H]\rightarrow U/H = U$, so that the map  $f\colon \vert [U/H](\mathbb R)\vert \rightarrow U(\mathbb R)$ is surjective. Proposition \ref{prop:inertia-description} implies the following description of the complex locus of the inertia stack $\ca I_{[U/H]}$ of $[U/H]$. 
\begin{lemma}\label{lemma:inertia:H/U}
In the above notation, we have
$$\vert \mathcal I_{[U/H]}(\mathbb C)\vert \simeq H(\mathbb C)/H(\mathbb C) \coloneqq  \{(p,h)\in U(\mathbb C)\times H(\mathbb C) \mid h \in H_p(\CC)\}/_\sim $$
where $\sim$ denotes the equivalence relation $(p,h)\sim (p',h')$  if $p=p'$ and $h$ is conjugate to $h'$ in $H_p(\mathbb C)$.
In particular, when $H$ is abelian, one has
$\vert \mathcal I_{[U/H]}(\mathbb C)\vert \simeq H(\mathbb C).$ \hfill \qed
\end{lemma}
\subsection{Topological and algebraic fundamental groups}\label{sec:topalgU/H}
\subsubsection{Action of fundamental groups.} Fix $p\in U(\mathbb R)$, and write $C$ for the connected component of $U(\mathbb R)$ containing $p$.
The finite \'etale group scheme $H \to U$ corresponds to an action $\pi^{\et}_1(U,\overline p)$ on $H_{\overline p}$, 
$$\rho_p \colon \pi^{\et}_1(U,\overline p)\rightarrow \Aut(H_{\overline p}),$$ 
which compatible with the group structure of $H_{\overline p}$. Here, $\pi^{\et}_1(U,\overline p)$ is the \'etale fundamental group of $U$ at the geometric point $\overline p$. 

Since $U$ is geometrically connected and $G = \pi_1^{\et}(\Spec(\RR))$, the natural morphisms $U_{\mathbb C}\rightarrow U\rightarrow \Spec(\mathbb R)$ induce a short exact sequence of groups
\begin{equation}\label{eq:homotopy exact sequence}
	1\rightarrow \pi^{\et}_1(U_{\mathbb C},\overline p)\rightarrow \pi^{\et}_1(U,\overline p)\rightarrow G\rightarrow 1.
\end{equation}
Restricting $\rho_p$ to $\pi^{\et}_1(U_{\mathbb C},\overline p)$, we get an action of $\pi^{\et}_1(U_{\mathbb C},\overline p)$ on $H_{\overline p}$,
$$\rho^{\mathbb C}_{p}\colon \pi^{\et}_1(U_{\mathbb C},\overline p)\rightarrow \Aut(H_{\overline p}),$$
which corresponds to the \'etale $U_{\mathbb C}$-group scheme $H_{\mathbb C}\rightarrow U_{\mathbb C}$. 
Recall that $\pi^{\et}_1(U_{\mathbb C},\overline p)$ identifies with the profinite completion of usual fundamental group $\pi_1(U(\mathbb C),p)$ so that, in particular, there is a natural map
$\pi_1(U(\mathbb C),p)\rightarrow \pi^{\et}_1(U_{\mathbb C},\overline p)$. We denote again by $$\rho^{\mathbb C}_{p} \colon \pi_1(U(\mathbb C), p) \to \Aut(H_{\overline p})$$ the restriction of  $\rho^{\mathbb C}_{p}\colon \pi^{\et}_1(U_{\mathbb C},\overline p)\rightarrow \Aut(H_{\overline p})$ along the map $\pi_1(U(\mathbb C), p) \to \pi_1^\et(U_\CC, \overline p)$; this representation of $\pi_1(U(\mathbb C), p)$ corresponds to the topological covering  $H(\mathbb C)\rightarrow U(\mathbb C)$. 

Viewing $p$ as a morphism of schemes $p \colon \Spec(\mathbb R)\rightarrow U$, we get a morphism  $\pi_1(p)\colon  G=\pi_1^\et(\Spec(\mathbb R))\rightarrow \pi^{\et}_1(U,\overline p)$
which splits (\ref{eq:homotopy exact sequence}), and hence yields an isomorphism
\begin{equation}\label{eq:semidirect}
\pi^{\et}_1(U,\overline p)\simeq \pi^{\et}_1(U_{\mathbb C},\overline p)\rtimes G. 
\end{equation}
Hence, viewing $G$ as a subgroup $G \subset \pi_1^\et(U,\overline p)$ via the isomorphism \eqref{eq:semidirect}, this induces an action of $G = \langle \sigma \rangle $ on $\pi^{\et}_1(U_{\mathbb C},\overline p)$ by the usual formula $\sigma \cdot \alpha = \sigma \alpha \sigma^{-1}$ for $\alpha \in \pi^{\et}_1(U_{\mathbb C},\overline p)$. This action is compatible with the natural action of $G$ on $\pi_1(U(\CC), p)$ defined as follows: for $\alpha\in \pi_1(U(\mathbb C),p)$, we have $\sigma \cdot \alpha = \sigma_{U\ast}(\alpha)$.

Restricting $\rho_p$ to $G$ via $\pi_1(p)$, we get an action of $G$ on $H_{\overline p}$ which corresponds to the natural involution $\sigma_p$ on $H_{\overline p}$, see \eqref{align:involution-G-Hx}. Consider the morphism $\pi_1(C,p)\rightarrow \pi_1(U(\mathbb C),p)$ induced by the embedding $C \subset U(\CC)$, and define \begin{align}\label{align:rhop-action}\rho_p^{C} \colon \pi_1(C,p) \to \Aut(H_{\overline p})\end{align} 
as the composition of $\rho_p^{\mathbb C}$ with $\pi_1(C,p) \to \pi_1(U(\CC),p)$. 

\begin{lemma}\label{lem:actionrestriction}
The above action \eqref{align:rhop-action} of $\pi_1(C,p)$ on $H_{\overline p}$ commutes with $\sigma_p$, in the sense that $\sigma_p(\gamma \cdot x) = \gamma \cdot \sigma_p(x)$ for $\gamma \in \pi_1(C,p)$ and $x \in H_{\overline p}$. In particular, it preserves the subset $\rm Z^1(G, H_{\overline p}) = \set{x \in H_{\overline p} \mid x \cdot \sigma_p(x) = e} \subset H_{\overline p}$ and the induced action of $\pi_1(C,p)$ on $\rm Z^1(G, H_{\overline p})$ descends to an action of $\pi_1(C,p)$ on $\rm H^1(G, H_{\overline p})$. 
\end{lemma}

\proof
We need to show that for every $\alpha\in \pi_1(C,p)$, one has   
\begin{align}\label{align:equation:needtoprove}\rho_p^{C}(\alpha) \circ \sigma_p=\sigma_p\circ \rho_p^{C}(\alpha) \quad \quad \text{as maps} \quad \quad H_{\overline p} \to H_{\overline p}.\end{align}
Via the isomorphism (\ref{eq:semidirect}), we write each element $\beta \in \pi_1^{\et}(U, \overline p)$ as a pair $\beta = (\beta_1, \beta_2)$ with $\beta_1 \in \pi_1^\et(U_\CC, \overline p)$ and $\beta_2 \in G$. Denoting again by $\alpha$ the image of $\alpha$ in $\pi^{\etale}_1(U_{\mathbb C},\overline p)$, equation \eqref{align:equation:needtoprove} can be rewritten as
\begin{align} \label{align:rewrite}\rho_p(\alpha, e)^{-1} \circ \rho_p(e, \sigma) \circ  \rho_p(\alpha, e)=\rho_p(e, \sigma),\end{align}
where $e$ is the neutral element of $\pi^{\etale}_1(U_{\mathbb C},\overline p)$. 
Since $\rho_p$ is a group homomorphism, we have
$$\rho_p(\alpha, e)^{-1}\circ \rho_p(e, \sigma)\circ  \rho_p(\alpha, e)=\rho_p((\alpha^{-1},e)\cdot(e, \sigma)\cdot (\alpha, e)).
$$
The semi-direct product group structure gives 
$(\alpha^{-1},e)\cdot (e, \sigma)\cdot (\alpha, e)=(\alpha^{-1}\sigma^{-1}\alpha \sigma,\sigma).$ The image $\alpha \in \pi_1(U(\CC),p)$ of $\alpha \in \pi_1(C,p)$ satisfies $\sigma_{U\ast}(\alpha) = \sigma_U \circ \alpha = \alpha$. 
For the image $\alpha \in \pi_1^\et(U_\CC, \overline p)$, one therefore has $\sigma \cdot \alpha = \sigma^{-1} \alpha \sigma = \alpha$.  
Hence, we get 
$\rho_p(\alpha, e)^{-1}\circ \rho_p(e, \sigma)\circ \rho_p(\alpha, e)=\rho_p((\alpha^{-1},e)(e, \sigma)(\alpha, e))=\rho_p(e,\sigma),$ so that \eqref{align:rewrite} follows.
\endproof
\subsubsection{Change of base point.}\label{sec:changeofbasepoint}
In the previous section, we fixed a point \( p \in U(\mathbb{R}) \) to study the fiber \( H_p \). For concrete calculations, it will be important to understand how \( H_p \) varies when the point $p$ varies in $U(\RR)$. This is the content of the following proposition. 
\begin{proposition}\label{prop:changeofbasepoint}
Let $U$ be a geometrically connected scheme of finite type over $\RR$ with $U(\RR) \neq \emptyset$. Let $Y\rightarrow U$ be a finite \'etale cover. 
Let $p,q\in U(\mathbb R)$ and choose a topological path $\gamma_{q,p}$ from $q$ to $p$ in $U(\mathbb C)$. Define $\omega_{q,p} \coloneqq (\sigma_U \circ \gamma_{q,p}) \ast \gamma^{-1}_{q,p}\in \pi_1(U(\mathbb C),p)$ (where $\ast$ denotes the composition of paths). Then the following diagram commutes:
\[
\xymatrixcolsep{3pc}
\xymatrix{
Y_{\overline q} \ar[d]^-{(\gamma_{q,p})_\ast}\ar[r]^-{\sigma_q} & Y_{\overline q} \ar@{=}[r] & Y_{\overline q}\ar[d]^-{(\gamma_{q,p})_\ast} \\
Y_{\overline p} \ar[r]^{\sigma_p}& Y_{\overline p} \ar[r]^-{\omega_{q,p}} & Y_{\overline p}.
}
\]
Here, $\omega_{q,p}$ acts on $Y_{\overline p}$ as an element of $\pi_1(U(\mathbb C),p)$, and $(\gamma_{q,p})_\ast \colon Y_{\overline q} \xrightarrow{\sim} Y_{\overline p}$ is the canonical isomorphism induced by the path $\gamma_{q,p}$. 
\end{proposition}
\proof
For any path $\gamma_{q,p} \colon [0,1] \to U(\CC)$ from $q$ to $p$, and any point $y \in Y_{\overline q}$, we let $\widetilde{\gamma}^{y}_{q,p}$ be the unique path in $Y(\mathbb C)$ that lifts $\gamma_{q,p}$ and that satisfies $\widetilde \gamma_{q,p}^y(0) = y$. 
This yields an isomorphism
\[
(\gamma_{q,p})_\ast  \colon Y_{\overline q} \xrightarrow{\sim} Y_{\overline p}, \quad \quad y \mapsto \widetilde \gamma_{q,p}^y(1). 
\]
By construction, we have:
$$\sigma_Y((\gamma_{q,p})_\ast (y))=\sigma_Y(\widetilde{\gamma}_{q,p}^y(1))\quad \text{and} \quad \omega_{q,p} \cdot  (\gamma_{q,p})_\ast (\sigma_Y(y))=\omega_{q,p}\cdot  \widetilde{\gamma}^{\sigma_Y(y)}_{q,p}.$$ 
Observe that
$$\sigma_Y(\widetilde{\gamma}_{q,p}^y(1))=\left(\sigma_{U\ast}(\widetilde{\gamma}_{p,q}^y)\right)(1)$$
and that $\sigma_{U\ast}(\widetilde{\gamma}_{p,q}^y)$ is a path in $Y(\mathbb C)$ that lifts $\sigma_{U\ast}(\gamma_{q,p})$ and that starts at $\sigma_Y(y)$.
In other words, 
$$\sigma_Y\left((\gamma_{q,p})_\ast (y)\right)=\widetilde{\sigma_{U\ast}(\gamma_{q,p})}^{\sigma_Y(y)}(1).$$
On the other hand, by construction of the action of $\pi_1(U(\mathbb C),p)$ on $Y_{\overline p}$, one has 
$$\omega_{q,p} \cdot\left(  \widetilde{\gamma_{q,p}}^{\sigma_Y(y)}(1)\right)=\widetilde{\omega_{q,p} \ast \gamma_{q,p}}^{\sigma_Y(y)}(1).$$
But $\omega_{q,p}\ast \gamma_{q,p}=\sigma_{U\ast}(\gamma_{q,p})$ by definition of $\omega_{q,p}$, hence the proof is concluded. \endproof
\begin{example} \label{ex:prop-apply-elliptic}
	Let $U\subset \mathbb G_m$ be an open subset whose real part contains $[-1,0)$ and $(0,1]$ and let $\pi \colon E\rightarrow U$ be a family of elliptic curves over $\RR$. Let $p=1\in U(\mathbb R)$ and assume that $Y_{p}$ is a maximal real elliptic curve (that is, that $Y_p(\RR)$ has two connected components). Consider the local system
    $
    \pi_\ast\ZZ/2
    $
    of finite dimensional $\ZZ/2$-modules on $U_{\et}$, and let $Y \to U$
    be the associated finite \'etale cover. Thus, 
    \[
    Y_{\overline q} = \rm H^1(E_{q}(\CC),\ZZ/2) \quad \quad \text{for} \quad \quad q \in U(\RR).
    \]
    Since $E_{ p}$ is a maximal real elliptic curve, the $G$-action on $Y_{\overline p} = \rm H^1(E_{p}(\CC),\ZZ/2)$ is trivial.
\begin{enumerate}
	\item Assume that the action of the standard loop $\gamma$ around $0$ (viewed as an element of $\pi_1^{\et}(U_{\mathbb C},\overline p)$) on $\rm H^1(E_p(\CC),\ZZ/2)$ is not trivial (this happens for example for the family whose affine equation is $y^2=(x^2-t)(x+2)$ where $t$ is the coordinate of $U$).
	Let $q=-1$ and choose as $\gamma_{q,p}$ the standard "half circle" around $0$, so that $\omega_{q,p}=\gamma$, hence it acts non-trivially on $\pi_1^{\et}(U_{\mathbb C},\overline p)$. Since the action of $G$ on $\rm H^1(E_p(\CC),\ZZ/2)$ is trivial, we deduce from Proposition \ref{prop:changeofbasepoint} that the action of $G$ on $\rm H^1(E_q(\CC),\ZZ/2)$ is not trivial. In particular, the real elliptic curve $E_{q}$ is not maximal (that is, $Y_q(\RR)$ is connected).  
	\item Assume that the action of $\pi_1^{\et}(U_{\mathbb C},\overline p)$ on $\rm H^1(E_p(\CC),\ZZ/2)$ is trivial (this happens for example for the family whose affine equation is $y^2=x(x+t)(x+2)$ where $t$ is the coordinate of $U$).
Let $q=-1$. Since $\pi_1^\et(U_{\CC},\overline p)$ acts trivially on $\rm H^1(E_{\overline p}(\CC),\ZZ/2)$, for every choice of path $\gamma_{q,p}$ from  $-1$ to $1$, the loop $\omega$ acts trivially on  $\rm H^1(E_{ p}(\CC),\ZZ/2)$, so that from Proposition \ref{prop:changeofbasepoint}, we deduce that $E_q$ is maximal.  
\end{enumerate}
\end{example}

\subsection{Proof of Theorem \ref{thm:topcovering[U/H]} and its corollaries.}\label{sec:proofU/H} We are ready to proceed with the:
\begin{proof}[Proof of Theorem \ref{thm:topcovering[U/H]}]
Item \ref{itemthm:UH:2} is the content of Lemma \ref{lem:actionrestriction}, while item \ref{item:3:theoremUH} is a particular case of Proposition \ref{prop:changeofbasepoint}. 
Hence we are left to prove item \ref{itemthm:UH:1}. By \cite[Theorem 1.5]{AdGF-grup}, the morphism $f^{-1}(C)\rightarrow C$ is a topological covering. Moreover, it has fiber over $p$ isomorphic to $H^1(G,H_p(\mathbb C))$ by Proposition \ref{prop:fiber-H1}.
We need to show that the action of $\pi_1(C,c)$ on $f^{-1}(p)$ corresponds, under the isomorphism $f^{-1}(p)\simeq H^1(G,H_p(\mathbb C))$, to the action defined in item \ref{itemthm:UH:2}. 
For this, set 
$$Z^1_C\coloneqq\set{(u, g) \in C \times H(\mathbb C) \mid  g \in H_u(\mathbb C)\text{ and } g \sigma(g) = e}\subset C\times H(\mathbb C).$$
Recall from \cite[Theorem 5.9]{AdGF-grup} that one has a canonical homeomorphism $f^{-1}(C)\simeq Z^1_C/_\sim $ where $(u, g) \sim (u',g')$ if $u = u'$ and there exists $h \in H_u(\mathbb C)$ such that $g' = h g \sigma(h)^{-1}$. 
Hence, one has the following commutative diagram in which the vertical arrows are covering maps and the square on the right is cartesian:
\begin{center}
	\begin{tikzcd}
f^{-1}(C)\arrow{d}\ar[d]& Z^1_C\arrow{d}\arrow[two heads]{l}\arrow[hook]{r} & H(\mathbb C)_{\vert C}\arrow{d}\arrow[hook]{r}\arrow[phantom]{rd}{\Box} & H(\mathbb C)\arrow{d}\\
C\arrow[equal]{r}& C\arrow[equal]{r} & C\arrow[hook]{r} & U(\mathbb C).
			\end{tikzcd}
	\end{center} 
Since the upper left horizontal arrow is surjective, this concludes the proof of item \ref{itemthm:UH:1}, and thereby of Theorem \ref{thm:topcovering[U/H]}. 
\end{proof}
\begin{proof}[Proof of Corollary \ref{cor:gerbesovercurves}]
By Lemma \ref{lemma:inertia:H/U},  $\vert \mathcal I_{\mathcal X}(\mathbb C)\vert $ is homeomorphic to $H(\mathbb C)$, so we need to prove that 
\[
h^\ast(\va{[U/H](\RR)}) \leq h^\ast(H(\CC)),
\]
Consider the diagram of finite  topological coverings
\[
\xymatrix{
\va{[U/H](\RR)} \ar[r]^-f &U(\RR)  & H(\RR). \ar[l]_-g
}
\]
We claim that there exists a canonical embedding of sheaves \begin{align} \label{align:label-embeddingofsheaves}f_\ast\ZZ/2 \hookrightarrow g_\ast\ZZ/2. \end{align} To prove this, we may fix a connected component $C \subset U(\RR)$, and define the embedding after restricting both sheaves to $C$. Fix $p \in C$. 
Writing $A = H_{\overline p} = A$, we have  that $A = B \times M$ is a product of two $G$-modules $B$ and $M$, where $B$ has underlying abelian group $(\ZZ/2)^n$ for some $n \geq 0$ and $M[2] = 0$. Let $\sigma \colon A \to A$ be the involution corresponding to the $G$-module structure of $A$. On the one hand, the \'etale cover $g \colon H(\RR)|_C \to C$ has fiber $g^{-1}(p)= A^\sigma$ and corresponds to the action of $\pi_1(C,p)$ on $A^\sigma$ induced by the map $\pi_1(C,p) \to \pi_1(U(\CC), p)$ and the action of $\pi_1(U(\CC),p)$ on $A$ (which is, in turn, induced by the \'etale cover $H \to U$). On the other hand, the \'etale cover $f^{-1}(C) \to C$ corresponds by Theorem \ref{thm:topcovering[U/H]} to the action of $\pi_1(C,p)$ on $f^{-1}(p) = \rm H^1(G, A)$ induced by the morphism $\pi_1(C,p) \to \pi_1(U(\CC),p)$ and the action of $\pi_1(U(\CC),p)$ on $A$. Since $B$
is an $\mathbb F_2$-vector space
and $M$ has no non-trivial $2$-torsion, one has 
$\rm H^1(G,A) = \rm H^1(G,B)= B^\sigma/(1+\sigma)B$. 
By the above, the quotient map 
\begin{align} \label{align:map-quot}
A^\sigma = B^\sigma \times M^\sigma \to B^\sigma/(1+\sigma)B = \rm H^1(G,A)
\end{align}
is $\pi_1(C,p)$-equivariant. Applying the contravariant functor $(\ZZ/2)^{(-)}$ to \eqref{align:map-quot} gives a $\pi_1(C,p)$-equivariant embedding
\[
(f_*\ZZ/2)_p = (\ZZ/2)^{\rm H^1(G,A)} \hookrightarrow (\ZZ/2)^{A^\sigma} = (g_*\ZZ/2)_p.
\]
Since the sheaves $f_*\ZZ/2$ and $g_*\ZZ/2$ are locally constant, this yields the embedding \eqref{align:label-embeddingofsheaves}. 

By taking global sections of the embedding of sheaves \eqref{align:label-embeddingofsheaves}, we obtain the inequality
\begin{equation}\label{eq:inequalitycurvesgerbes1}
\dim \rm H^0(\va{[U/H](\RR)}, \ZZ/2) \leq \dim \rm H^0(H(\RR),\ZZ/2). 
\end{equation}
Since $f$ and $g$ are finite maps, we have $f_*=f_{!}$ and $g_*=g_{!}$. Hence, by taking global sections with compact support of the embedding \eqref{align:label-embeddingofsheaves}, we get 
\begin{equation}\label{eq:inequalitycurvesgerbes2}
\dim \rm H^0_c(\va{[U/H](\RR)}, \ZZ/2) \leq \dim \rm H_c^0(H(\RR),\ZZ/2).
\end{equation}
All the connected components of $U(\mathbb R)$ are open intervals or circles; as $f\colon\vert [U/H](\mathbb R)\vert \rightarrow U(\mathbb R)$ is a finite topological covering by Theorem \ref{thm:topcovering[U/H]}, this also holds for $\vert [U/H](\mathbb R)\vert$. Hence, $\dim H^1(K, \ZZ/2) = \dim H_1(K,\ZZ)$ for each such a connected component $K$, and the latter equals $\dim H^0_c(K,\ZZ/2)$ by Poincaré duality. Therefore, one has:
\begin{align*}h^\ast(\va{[U/H](\RR)})&=\dim \rm H^0(\va{[U/H](\RR)}, \ZZ/2)+\dim \rm H_c^0(\va{[U/H](\RR)}, \ZZ/2)\\
&\leq \dim \rm H^0(H(\RR), \ZZ/2)+\dim \rm H_c^0(H(\RR), \ZZ/2)\\
&=h^\ast(H(\RR)),
\end{align*}
where the inequality follows from (\ref{eq:inequalitycurvesgerbes1}) and (\ref{eq:inequalitycurvesgerbes2}). The classical Smith--Thom inequality \eqref{align:inequality-ST}
implies $h^\ast( H(\RR)) \leq h^\ast(H(\CC))$, so 
$
h^\ast(\va{[U/H](\RR)}) \leq h^\ast( H(\RR))) \leq h^\ast(H(\CC))$. 
\end{proof}

\begin{proof}[Proof of Corollary \ref{cor:gerbesovercurvesconcrete}]
Since $U$ is a smooth proper curve, by Theorem \ref{thm:topcovering[U/H]}, $\vert [U/H](\mathbb R)\vert $ is a disjoint union of circles, so that
$$h^{*}(\vert [U/H](\mathbb R)\vert )=2\cdot\#\pi_0(\va{[U/H](\mathbb R)}).$$
By Theorem \ref{thm:topcovering[U/H]}, for every connected component $C_i\in \pi_0( \vert [U/H](\mathbb R)\vert )$, the map
$f^{-1}(C_i)\rightarrow C_i$ is the \'etale cover corresponding to the action of $\pi_1(C_i,p_i)$ on $\rm H^1(G,H_{\overline p_i})$. Hence, $f^{-1}(C_i)$ has $\# \big(\rm H^1(G,H_{\overline p_i})/\pi_1(C_i,p_i)\big) $ connected components. Consequently, we get 
\begin{align}\label{align:descr-UH-R}h^{*}(\vert [U/H](\mathbb R)\vert)=\sum_{i = 1}^m 2\cdot\# \left(\frac{\rm H^1\left(G,H_{\overline p_i}\right)}{\pi_1(C_i,p_i)}\right).\end{align}
Assuming Conjecture \ref{conj:ST}, we have \begin{align} \label{align:proof-ST-UH}h^{\ast}(\vert [U/H](\mathbb R)\vert) \leq h^{*}(\vert \ca I_{[U/H]}(\mathbb C)\vert ).
\end{align}
Therefore, to finish the proof, it suffices to show that
\begin{align}\label{eq : concretesmiththom}
\begin{split}
	h^{*}(\vert \ca I_{[U/H]}(\mathbb C)\vert )&= 2\cdot\#\left(
    \frac{H_{\overline p}}{\pi_1(U(\mathbb C), p)}
    \right)\\&+\sum_{D\in (H_{\overline p}/\pi_1(U(\mathbb C), p))} 2\cdot\big (\# D \cdot(g-1) +1\big ),
    \end{split}
\end{align}
for indeed, by the fact that $\sum_D \# D = \# H_{\overline p}$, where the sum ranges over the orbits $D\in (H_{\overline p}/\pi_1(U(\mathbb C), p))$, equations \eqref{align:descr-UH-R}, \eqref{align:proof-ST-UH} and \eqref{eq : concretesmiththom} together imply equation \eqref{align:inequality:H/U}. 

By Lemma \ref{lemma:inertia:H/U}, the map $\vert \ca I_{[U/H]}(\mathbb C)\vert = H(\CC) \rightarrow \vert \mathcal [U/H](\mathbb C)\vert=U(\mathbb C) $ is the \'etale cover corresponding to the action of $\pi_1(U(\mathbb C),p)$ on $H_p(\CC)$. Hence, $\vert \ca I_{[U/H]}(\mathbb C)\vert $ has $\#(H_{\overline p}/\pi_1(U(\mathbb C),p))$ connected components, and each of these is a connected \'etale cover whose degree equals the cardinality of the orbit. Thus, (\ref{eq : concretesmiththom}) follows from the Riemann--Hurwitz formula, and the proof is finished.
\end{proof}

\section{Gerbes and the homotopy exact sequence} \label{sec:splitting}

In this section, we refine the results of Section \ref{sec:classifyingstack}, in order to be able to apply them in explicit examples. A key step is the reinterpretation of Proposition \ref{prop:changeofbasepoint} in terms of splittings of the homotopy exact sequence \eqref{eq:homotopy exact sequence}. This reinterpretation is developed in Section \ref{sec:topopathsplitting} and subsequently applied in Section \ref{sec:exenriques} to the study of certain gerbes over an Enriques surface.

We retain the notation introduced in Section \ref{sec:classifyingstack}. In particular, \( U \) denotes a geometrically connected scheme locally of finite type over \( \mathbb{R} \) with \( U(\mathbb{R}) \neq \emptyset \), and \( H \to U \) is a finite étale group scheme over \( U \).

\subsection{Topological paths and splitting of the homotopy exact sequence}\label{sec:topopathsplitting}
\subsubsection{Galois formalism.}
Write $\mathrm{Fset}$ for the category of finite sets and, for a scheme $Z$, $\rm{F}\etale(Z)$ for the category of finite \'etale covers of $Z$.
 Recall that for every $q\in U(\mathbb R)$, the group 
$\pi^{\etale}_1(U,\overline q)$ (resp.\ $\pi^{\etale}_1(U_{\mathbb C},\overline q)$) is  the automorphism of the functor
$$(-)_{\overline q}:\rm{F}\etale(U)\rightarrow \mathrm{Fset}\quad (\text{resp.\ } (-)^{\mathbb C}_{\overline q}:\rm{F}\etale(U_{\mathbb C})\rightarrow \mathrm{Fset})$$
sending $Y\rightarrow U$ (resp.\ $Y\rightarrow U_{\mathbb C}$) to the geometric fiber $Y_{\overline q}$. 
By the general formalism of Galois categories,  every isomorphism of functors
$\phi \colon (-)_{\overline q}\xrightarrow{\simeq} (-)_{\overline p}$ induces an isomorphism, well defined up to conjugation, 
$\varphi\colon \pi_1^\et(U,\overline q)\xrightarrow{\simeq} \pi_1^\et(U,\overline p),$
in such a way that the action of $G$ on $H_{\overline q}$
is induced by the action of $\pi_1^\et(U,\overline p)$ on $H_{\overline p}$ and the composition
$$G=\pi_1^\et(\Spec(\mathbb R))\xrightarrow{\pi_1(q)}\pi_1^\et(U,\overline q)\xrightarrow{\varphi}\pi_1^\et(U,\overline p).$$
\begin{remark}\label{rmk:nonuniques}
    Let us emphatize that, in general, the isomorphism $\varphi \colon \pi_1^\et(U,\overline q)\xrightarrow{\simeq}\pi_1^\et(U,\overline p)$ induced by an isomorphism of functors $\phi \colon (-)_{\overline q}\xrightarrow{\simeq} (-)_{\overline p}$ is not unique. It is only unique up to conjugation.
\end{remark}
Since both $\pi_1(p)$ and $\varphi\circ\pi_1(q)$ are splittings of the exact sequence 
$$0\rightarrow \pi_1^\et(U_{\mathbb C},\overline p)\rightarrow \pi_1^\et(U,\overline p)\rightarrow G\rightarrow 0,$$
to understand the action of $G$ on $H_{\overline q}$, one has to understand how the different splittings of this sequence are related. This is the main result of the section. 

\subsubsection{Splittings of the homotopy exact sequence.} Let $p,q \in U(\RR)$. Let $\gamma_{q,p} \colon [0,1] \to U(\CC)$ be a path from $q$ to $p$. 
The isomorphisms $(\gamma_{q,p})_\ast \colon Y_{\overline q}\simeq Y_{\overline p}$ induced by $\gamma_{q,p}$ fit together to give an isomorphism 
$\varphi^{\mathbb C}_{q,p}\colon (-)^{\mathbb C}_{\overline q}\xrightarrow{\simeq}(-)^{\mathbb C}_{\overline p}$ of fiber functors. In turn, this induces an isomorphism $\varphi^{\mathbb C}_{q,p}\colon \pi^{\et}_1(U_{\mathbb C},\overline q)\xrightarrow{\simeq }\pi^{\et}_1(U_{\mathbb C},\overline p),$ well defined up to conjugation, extending the usual isomorphism $\pi_1(U(\mathbb C),q)\rightarrow \pi_1(U(\mathbb C),p)$ defined by $\alpha\mapsto \gamma_{q,p}\alpha \gamma_{q,p}^{-1}$. Write $$\omega \coloneqq \sigma_{U\ast}(\gamma_{q,p})\ast \gamma_{q,p}^{-1} \in \pi_1(U(\CC), p),$$ 
where $\ast$ is the concatenation of paths. By abuse of notation, let $\omega \in \pi_1^\et(U_\CC, \overline p)$ be the image of $\omega$ under the natural morphism $\pi_1(U(\CC), p) \to \pi_1^\et(U_\CC, \overline p)$. 
\begin{proposition}\label{prop:splitting}
In the above notation, consider the map $f \colon G \to \pi_1^\et(U_\CC, \overline p) \rtimes G$ defined as the 
composition 
$$
f \colon G \xlongrightarrow{\pi_1(q)}
\pi_1^\et(U,\overline q) \xlongrightarrow{\varphi_{q,p}} \pi^{\et}_1(U,\overline p)\simeq \pi^{\et}_1(U_{\mathbb C},\overline p)\rtimes G,$$	
	where the isomorphism on the right is defined by the splitting of the homotopy exact sequence (\ref{eq:homotopy exact sequence}) induced by the section $\pi_1(p) \colon G \to \pi_1^\et(U,\overline p)$. Then $f(\sigma)$ is conjugate to $(\omega, \sigma)$. 
\end{proposition}
\begin{proof}
Let $Y\rightarrow U$ be a finite connected \'etale cover. 
By Proposition \ref{prop:changeofbasepoint}, 
the action of $G=\pi_1^\et(\Spec(\mathbb R))$ on 
$Y_{\overline p}$ 
induced by the action of $\pi_1^\et(U,\overline p)$ on $Y_{\overline p}$, and up to multiplying by $\omega$, the composition
$$G\xrightarrow{\pi_1(q)}\pi_1^{\etale}(U,\overline q)\xrightarrow{\varphi_{q,p}} \pi_1^{\etale}(U,\overline p)$$
identifies with the natural action of $G$ on $Y_{\overline p}$. To be precise, after identifying 
$Y_{\overline p}$ and $Y_{\overline q}$ using $\gamma_{q,p}$, one has $\sigma_q=\omega\cdot  \sigma_p$. 
The proposition follows from this. 
\end{proof}

\begin{examples}
    Let \( U = \mathbb{G}_m \) and take \( p = 1 \in \mathbb{G}_m(\mathbb{R}) \). In this case, the étale fundamental group is given by $\pi_1^{\etale}(U, \overline p) \simeq \hat{\mathbb{Z}} \rtimes G,
    $ where \( G \) acts on \( \hat{\mathbb{Z}} \) by inversion.
    
    \begin{enumerate}
        \item Take \( q = 2 \) and choose \( \gamma_{q,p} \) as the natural path contained in the real part connecting \( q \) to \( p \). In this case, \(\omega\) is the trivial loop, so the image of the section corresponding to \( \pi_1(q) \) is identified with \( (0,e) \).
        
        \item Again, take \( q = 2 \), but this time let \( \gamma_{q,p} \) be the path shown on the left in Figure \ref{fig: paths}, so that \(\omega\) is nontrivial. By construction, the class of \(\omega\) in \(\pi_1(U(\mathbb{C}), p) \simeq \mathbb{Z}\) is \( 2 \), so under this choice of \( \gamma \), the image of the section \( \pi_1(q) \) is \( (2,e) \). Although different from the previous case, we note that \( (2,e) \) is conjugate to \( (0,e) \) in \( \mathbb{Z} \rtimes G \), meaning that the conjugacy class of the section remains unchanged (see Remark \ref{rmk:nonuniques}). 
        
        \item Take \( q = -1 \) and choose \( \gamma_{q,p} \) as the "half-circle path" from \( -1 \) to \( 1 \), as depicted on the right in Figure \ref{fig: paths}. In this case, \( \gamma \) corresponds to the class of \( 1 \) in \( \mathbb{Z} \), so the image of the section \( \pi_1(q) \) is \( (1,e) \). Since \( (1,e) \) is not conjugate to \( (0,e) \), this section is genuinely different from the previous ones, even up to conjugation.
    \end{enumerate}
 \begin{figure}[h!]
        \begin{center}
            \begin{picture}(0,120)
                \put(-134,5){$0$}
                \put(-85,5){$1$}
                \put(-30,5){$2$}
                \put(-150,-5){\includegraphics[width=0.3\textwidth]{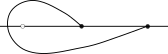}}
                \put(44,5){$-1$}
                \put(97,5){$0$}
                \put(136,5){$1$}
                \put(25,15){\includegraphics[width=0.3\textwidth]{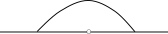}}
            \end{picture}
            \caption{Two paths in $\bb G_m(\CC)$.}
            \label{fig: paths}
        \end{center}
    \end{figure}
\end{examples}
\subsection{Smith--Thom for various split gerbes over an Enriques surface.}\label{sec:exenriques}
In this section, we apply the theory developed in the previous section -- particularly Theorem \ref{thm:topcovering[U/H]} -- to verify the Smith–Thom inequality for certain gerbes over an Enriques surface.

Let \( U \) be an Enriques surface such that \( U(\mathbb{R}) \neq \emptyset \), so that its K3 cover \( h \colon V \to U \) is defined over \( \mathbb{R} \). To simplify the discussion, we also assume that \( V(\mathbb{R}) \neq \emptyset \).  
Fix a point \( p \in U(\mathbb{R}) \) in the image of \( h \colon V(\mathbb{R}) \to U(\mathbb{R}) \). Recall that $\pi_1^{\et}(U_\CC,\overline p)\simeq \ZZ/2$, so that, since $\ZZ/2$ has no non-trivial automorphisms, the section \( \pi_1(p) \) induces an isomorphism 
$\pi_1^{\et}(U, \overline p) \simeq \pi_1^{\et}(U_\CC,\overline p) \times G$ such that the K3 cover \( h \colon V \to U \) corresponds to the morphism
\[
\pi_1 \colon \pi_1^\et(U_\CC,\overline p) \times G \to \mathbb{Z}/2\]
which is the composition of the projection onto the first factor and the isomorphism $\pi_1^\et(U_\CC,\overline p)\simeq \ZZ/2$. 

We let $1$ be the generator of $\pi_1^\et(U_\CC,\overline p)$. For every \( q \in U(\mathbb{R}) \), the group \( G \) acts trivially on \( V_{\overline q} \) if and only if \( q \) is in the image of \( h \colon V(\mathbb{R}) \to U(\mathbb{R}) \). Consequently, by Proposition \ref{prop:splitting}, the element \( \epsilon_q \) corresponding to the section associated with \( q \) is given by  
\[
\epsilon_q = 
\begin{cases}
(0,0) & \text{if } q \text{ is in the image of } h \colon V(\mathbb{R}) \to U(\mathbb{R}), \\
(1,0) & \text{otherwise}.
\end{cases}
\]  
For the remainder of this section, we let \(G= \mathbb{Z}/2 \) act on \( \mathbb{Z}/2 \oplus \mathbb{Z}/2 \) by exchanging the coordinates.
\begin{example}\label{ex: enriques 1}
    Assume that \( U(\mathbb{R}) \) is the union of four copies of \( \mathbb{P}^2(\mathbb{R}) \) and two spheres \( S^2 \), and that the map \( h \colon V(\mathbb{R}) \to U(\mathbb{R}) \) is surjective (such Enriques surfaces exist, as shown in \cite[Table 8, page 180]{itenberg-enriques}).  

    Let \( \pi_1 \colon \pi_1^\et(U_\CC,\overline p) \times G \to \mathbb{Z}/2 \) be the composition of the projection onto the first coordinate and the isomorphism $\pi_1^\et(U_\CC,\overline p)\simeq \ZZ/2$, and let \( \mathbb{Z}/2 \times G \) act on \( \mathbb{Z}/2 \oplus \mathbb{Z}/2 \) via this map. Denote by \( H \to U \) the corresponding group scheme. Since \( h \colon V(\mathbb{R}) \to U(\mathbb{R}) \) is surjective, for every \( q \in U(\mathbb{R}) \), the image of the section \( \pi_1(q) \) is \( (0,e) \). Consequently, the action of \( G \) on \( H_{\overline q} \) is trivial, implying that  
    $
    \mathrm{H}^1(G, H_{\overline q}) = \mathbb{Z}/2 \times \mathbb{Z}/2.
    $  
Let \( C \) be a connected component homeomorphic to \( S^2 \). Since \( S^2 \) is simply connected, the cover \( f^{-1}(C) \to C \) is trivial. Thus, the preimage of each such \( C \) under the map  
    $
    f \colon \lvert [U/H](\mathbb{R}) \rvert \to U(\mathbb{R})
    $  
    consists of four copies of \( S^2 \).  

    On the other hand, let \( C \) be a connected component homeomorphic to \( \mathbb{P}^2(\mathbb{R}) \). The natural map \( \pi_1(C) \to \pi_1(U(\mathbb{C})) \) is an isomorphism, as there is at least one spherical connected component in the real locus for one of the real structure on the K3 cover over \( C \) (see the discussion in \cite[Section 3.5]{itenberg-enriques}). Thus, the cover \( f^{-1}(C) \to C \) has three connected components, corresponding to the orbits  
    $
    \{(0,0)\}, \{(1,e)\}, \{(0,e), (1,0)\}
    $
    of the action of \( \pi_1(U(\mathbb{C})) \) on  
    $
    \mathrm{H}^1(G, H_{\overline q}) = \mathbb{Z}/2 \times \mathbb{Z}/2.
    $  
    Since the \( \pi_1(C) \)-action on \( \{(0,e), (1,0)\} \) is nontrivial, the corresponding cover is homeomorphic to the universal cover \( S^2 \to \mathbb{P}^2(\mathbb{R}) \). Hence, \( f^{-1}(C) \) is homeomorphic to the disjoint union of two copies of \( \mathbb{P}^2(\mathbb{R}) \) and one \( S^2 \).  

    In conclusion, \( \lvert [U/H](\mathbb{R}) \rvert \) is homeomorphic to the disjoint union of:
    \begin{itemize}
        \item four copies of \( S^2 \coprod \mathbb{P}^2(\mathbb{R}) \coprod \mathbb{P}^2(\mathbb{R}) \), each lying over a \( \mathbb{P}^2(\mathbb{R}) \), and
        \item two copies of \( \coprod_{1 \leq i \leq 4} S^2 \), each lying over an \( S^2 \).
    \end{itemize}
    In particular, we obtain  
    \[
    h^*(\lvert [U/H](\mathbb{R}) \rvert) = 4 \cdot (2+3+3) + 2 \cdot 8 = 48.
    \]  
By Lemma \ref{lemma:inertia:H/U}, the inertia \( \va{\ca I_{[U/H]}(\mathbb{C})} \to U(\mathbb{C}) \) corresponds to the cover associated with the action of \( \pi_1(U(\mathbb{C})) \) on \( H_{\overline p} \), which has three connected components, corresponding to the orbits  
    $
    \{(0,0)\}, \{(1,e)\},\{(0,e), (1,0)\}$.   
    Hence, \( \va{\ca I_{[U/H]}(\mathbb{C})} \) is the disjoint union of two copies of \( U(\mathbb{C}) \) and one copy of its K3 cover. In particular, $h^*(\va{\ca I_{[U/H]}(\mathbb{C})}) = 16 \cdot 2 + 24 = 56$. 
    Thus, the Smith–Thom inequality \eqref{align:inequality-ST-stacks} is verified for the stack $[U/H]$. 
\end{example}

\begin{example}
    Retaining the notation of Example \ref{ex: enriques 1}, we now assume that the image of the map \( V(\mathbb{R}) \to U(\mathbb{R}) \) consists of only three copies of \( \mathbb{P}^2(\mathbb{R}) \) and a single \( S^2 \) (such Enriques surfaces exist by \cite[Table 8, page 180]{itenberg-enriques}).  

    The description of the cover \( f^{-1}(C) \to C \) remains the same for the connected components in the image of \( V(\mathbb{R}) \to U(\mathbb{R}) \). However, it differs for the connected components \( C_1 \) and \( C_2 \), which are respectively homeomorphic to \( S^2 \) and \( \mathbb{P}^2(\mathbb{R}) \) but are not in the image. In the end, $\vert [U/H](\mathbb R) \vert$ has the following description.
        \begin{itemize}
        \item Three copies of \( S^2 \coprod \mathbb{P}^2(\mathbb{R}) \coprod \mathbb{P}^2(\mathbb{R}) \), each lying over a \( \mathbb{P}^2(\mathbb{R}) \);
                \item One copy of \( \coprod_{1 \leq i \leq 4} S^2 \), lying over an \( S^2 \);
        \item One copy of $\mathbb{P}^2(\mathbb{R})$, lying over an \( \mathbb{P}^2(\mathbb{R}) \) and one copy of \( S^2 \), lying over an \( S^2 \).
\end{itemize}
    This is illustrated in Figure \ref{fig: enriques} below, where the dark disks represent copies of \( \mathbb{P}^2(\mathbb{R}) \), and the light spheres represent copies of $S^2$.   

    \begin{figure}[h!]
        \begin{center}
            \begin{picture}(0,190)
                \put(-150,75){$\lvert [U/H](\mathbb{R}) \rvert$}
                \put(-150,15){$U(\mathbb{R})$}
                \put(-40,15){\includegraphics[width=0.3\textwidth]{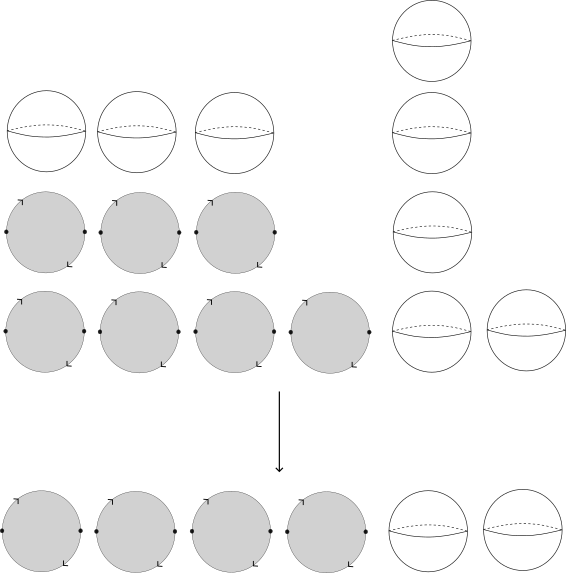}}
            \end{picture}
            \caption{The morphism \( \lvert [U/H](\mathbb{R}) \rvert \to U(\mathbb{R}) \)}
            \label{fig: enriques}
        \end{center}
    \end{figure}
    To justify this, choose a point \( q_i \in C_i \). Since \( q_i \) is not in the image of \( V(\mathbb{R}) \to U(\mathbb{R}) \), the action of \( G \) on \( V_{\overline q_i} \) is nontrivial. Consequently, the image of the section \( \pi_1(q_i) \) is \( (1,e) \), implying that \( G \) acts on \( H_{\overline q_i} \) by exchanging the coordinates. In particular,  
    $\mathrm{H}^1(G, H_{\overline q_i}) = 0$ is trivial, 
    so that \( f^{-1}(C_i) \to C_i \) is an isomorphism. As in the previous example, one verifies that the Smith–Thom inequality holds in this case as well.
\end{example}

\begingroup
\sloppy
\printbibliography
\endgroup

\vspace{5mm}

\textsc{Emiliano Ambrosi, Institut de Recherche Mathématique Avancée (IRMA),
7 Rue René Descartes, 67084 Strasbourg, France}\par\nopagebreak 
  \textit{E-mail address:} \texttt{eambrosi@unistra.fr}

\vspace{5mm}

\textsc{Olivier de Gaay Fortman, Department of Mathematics,
       Utrecht University,
       Budapestlaan 6, 3584 CD
       Utrecht, The Netherlands}\par\nopagebreak 
  \textit{E-mail address:} \texttt{a.o.d.degaayfortman@uu.nl}

\end{document}